\documentclass[10pt,twoside,a4paper]{amsart}

\usepackage[english]{babel}
\usepackage[utf8]{inputenc}

\usepackage{amsfonts,amsthm,amssymb,amsmath,amsthm,latexsym,wasysym,stmaryrd,mathrsfs,dsfont,txfonts,xcolor}
\usepackage{accents,multirow,rotating,subfigure,graphicx,bigstrut,lscape,multicol,enumerate,accents,mathtools,exscale}

\usepackage[normalem]{ulem}

\newcommand{\newmu}{\nu}

\usepackage{pdftricks}
\usepackage{color}
\usepackage[pdftex,all]{xy}

\usepackage{hyperref}
\usepackage{appendix}
\usepackage[justification=centering]{caption}

\usepackage{scrextend}

\graphicspath{{Figures/}}

\numberwithin{equation}{section}

\theoremstyle{theorem}

\newtheorem {theo}{Theorem}[section]
\newtheorem*{theo*}{Theorem}
\newtheorem {lemme}[theo]{Lemma}
\newtheorem*{lemme*}{Lemma}
\newtheorem {prop}[theo]{Proposition}
\newtheorem*{prop*}{Proposition}
\newtheorem {cor}[theo]{Corollary}
\newtheorem*{cor*}{Corollary}

\newtheorem*{cor_proof*}{Corollary (of the proof)}

\newtheorem*{conjecture*}{Conjecture}
\theoremstyle{definition}
\newtheorem {defi}[theo]{Definition}
\newtheorem*{defi*}{Definition}
\newtheorem {nota}[theo]{Notation}
\newtheorem*{nota*}{Notation}
\theoremstyle{remark}
\newtheorem {remarque}[theo]{Remark}
\newtheorem*{remarque*}{Remark}

\newtheorem*{warning*}{Warning}

\newtheorem*{remarques*}{Remarks}

\newtheorem*{warnings*}{Warnings}

\newtheorem*{convention*}{Convention}
\newtheorem {exemple}[theo]{Example}
\newtheorem*{exemple*}{Example}

\newtheorem*{exemples*}{Examples}

\newtheorem*{question*}{Question}

\newtheorem*{questions*}{Questions}

\newtheorem*{fact*}{Fact}
\newtheorem{claim}[theo]{Claim}
\newtheorem*{acknowledgments}{Acknowledgments}

\def\AA{{\mathcal A}}

\def\C{{\mathcal C}}
\def\CC{{\mathcal C}}

\def\N{{\mathds N}}

\def\R{{\mathds R}}

\def\UU{\mathcal C}

\def\Z{{\mathds Z}}

\def\2Z{{\fract{\Z}{2\Z}}}

\def\e{\varepsilon}
\def\p{\partial}

\newcommand{\fract}[2]{\hbox{\leavevmode 
\kern.1em \raise .25ex \hbox{\the\scriptfont0 $#1$}\kern-.1em }\big/
  {\hbox{\kern-.15em \lower .5ex \hbox{\the\scriptfont0 $#2$}} }}
\newcommand{\subfract}[2]{\hbox{\leavevmode
  \kern0em \raise .25ex \hbox{\the\scriptfont0 \tiny $#1$}\kern-.1em }/
  {\hbox{\kern-.15em \lower .5ex \hbox{\the\scriptfont0 \tiny $#2$}} }}

\newcommand{\dessin}[2]{
  \vcenter{\hbox{\includegraphics[height=#1]{#2.pdf}}}}

\newcommand{\function}[5]{
  #1 \colon
  \begin{array}{ccc}
    #2 & \longrightarrow & #4\\[.2cm]
    #3 & \longmapsto & #5
  \end{array}}

\renewcommand{\quote}[1]{`#1'}

\newcommand{\omu}{\overline{\mu}}

\def\nR{\textnormal R}

\DeclareMathOperator{\Spun}{Spun}

\DeclareMathOperator{\Ker}{Ker}
\DeclareMathOperator{\rk}{rank}

\renewcommand{\Im}{\textrm{Im}}

\definecolor{pink}{rgb}{0.858, 0.188, 0.478}
\definecolor{orange}{rgb}{1, 0.647, 0}
\definecolor{vert}{rgb}{0.42, 0.557, 0.137}

\begin{document} 

\title{Milnor-type invariants for surface-links and cut-diagrams} 
\author[B. Audoux]{Benjamin Audoux}
         \address{Aix Marseille Univ, CNRS, Centrale Marseille, I2M, Marseille, France}
         \email{benjamin.audoux@univ-amu.fr}
\author[J.B. Meilhan]{Jean-Baptiste Meilhan} 
\address{Univ. Grenoble Alpes, CNRS, Institut Fourier, F-38000 Grenoble, France}
	 \email{jean-baptiste.meilhan@univ-grenoble-alpes.fr}
\author[A. Yasuhara]{Akira Yasuhara} 
\address{Faculty of Commerce, Waseda University, 1-6-1 Nishi-Waseda,
  Shinjuku-ku, Tokyo 169-8050, Japan}
	 \email{yasuhara@waseda.jp}
\begin{abstract} 
We generalize Milnor link invariants to surface-links in $4$--space, possibly with boundary. 
To this end, we introduce the notion of cut-diagram, which is a  $2$--dimensional analogue of Gauss diagrams. 
To each cut-diagram, we associate a group extending the fundamental group of the exterior of a surface-link, and we extract Milnor-type invariants from its successive nilpotent quotients. 
We show that this yields concordance invariants for surface-links, and that some even are link-homotopy invariants.
We give several concrete applications, including realization and classification results. 
The theory of cut-diagrams is further investigated, heading towards a 
combinatorial approach to surfaces in $4$--space.
\end{abstract} 

\maketitle

\setcounter{tocdepth}{1}
\tableofcontents

\section*{Introduction}
The study of concordance of knots and links in the $3$-sphere has produced a wide range of invariants over the past decades, revealing subtle geometric and algebraic structures.  
The situation for $2$--links, \emph{i.e.} smooth embeddings of copies of $S^2$ in $S^{4}$, is quite different.  Although every $2$-knot is slice \cite{Kervaire}, it is still unknown whether there exist any non-slice $2$-links (see \cite[Problem 1.56]{Kirby}). 

For classical links, one of the most powerful families of concordance invariants is given by Milnor invariants. Introduced by Milnor in the 1950s \cite{Milnor,Milnor2}, these invariants generalize the linking number and capture higher-order linking phenomena through the nilpotent quotients of the fundamental group of the link complement. They have proved very effective in detecting nontrivial concordance classes.  

A natural and longstanding question is whether Milnor invariants admit meaningful extensions to higher-dimensional links. 
Nontrivial generalizations do exist in the related context of link maps, 
where the embedding condition is relaxed \cite{Kolink}, but these constructions do not directly address the embedded case, and rather concern the weaker relation of link-homotopy.  
In this paper, we propose a different approach. Rather than restricting attention to spherical links, we study general surface-links in 4–space, allowing components of arbitrary genus and possibly with boundary. 
Within this broader setting, we define Milnor-type invariants that extend the classical construction and provide new concordance invariants. 

In order to define our generalization of Milnor invariants, we develop a theory of cut-diagrams. 
A \emph{cut-diagram} consists of an oriented  diagram on a surface, endowed with some labeling. 
This provides a combinatorial framework for studying surface-links, which can be seen as a $2$--dimensional version of Gauss diagrams. 
In particular, any diagram of a surface-link naturally gives rise to a cut-diagram, but the definition is flexible enough to allow a purely abstract treatment.
To each cut-diagram, we associate a group generated by its \lq regions\rq\, and defined by Wirtinger-type relations. In the case of a cut-diagram arising from a surface-link, this group coincides with the fundamental group of the complement. A key step in our construction is the derivation of a Chen–Milnor type presentation for the nilpotent quotients of this group (Theorem \ref{alaChen}). This presentation, obtained by combinatorial methods, generalizes the classical result of Milnor for links \cite[Thm.~4]{Milnor2} and provides an effective tool for computation, see Remark \ref{rem:Stallings0}. 

Building on this, we define Milnor-type invariants by considering the Magnus expansion of suitable elements associated with paths on the underlying surface. These elements generalize the notion of preferred longitudes in classical link theory. Although the resulting coefficients depend a priori on several choices, we show how to extract well-defined numerical invariants.
More precisely, for each component of the surface-link and each sequence of indices in $\{1,\cdots,n\}$, where $n$ is the number of components of the underlying surface, we define two families of invariants: loop-invariants, associated with closed curves on the surface, and arc-invariants, associated with paths connecting boundary components. 
We prove the following (see Corollaries \ref{cor:MilnorConcordance} and \ref{cor:lhinvariance}).
\begin{theo*}
  Milnor loop- and arc-invariants are well-defined concordance invariants of surface-links. 
  Moreover, if the indexing sequence contains no repetition, the corresponding Milnor invariants are also invariant under link-homotopy. 
\end{theo*}

Since our invariants are extracted from essential loops and arcs on the surface, they all vanish for spherical $2$--links (Remark \ref{pourlesgolmons}).
However, as soon as at least one component is not simply connected, the invariants become nontrivial and, in many cases, quite effective.
This is illustrated in Section \ref{sec:appli}, where a number of applications are provided:   
\begin{itemize}
 \item We give a general realization result, showing the existence of infinitely many 
 surface-links consisting of a single torus and a union of spheres that are slice, but neither null-concordant nor link-homotopically trivial (Proposition~\ref{prop:anginedepoitrine}). In contrast, it is known that all $2$-links are link-homotopically trivial \cite{BT}.
\item In Section \ref{sec:Saito}, we consider a family $\left(W_m\right)_{m\in \mathbf{N}}$ of links made of a sphere and a torus. We show that $W_{m'}$ and $W_{m}$ are concordant if and only if $m=m'$ by using our Milnor invariants while, in contrast, none of the concordance invariants defined by Sato-Levine \cite{Sato}, Cochran \cite{Cochran} and Saito \cite{Saito} 
can distinguish $W_m$ and $W_{m'}$ for $m\neq m'$. 
   \item In Section \ref{sec:lhkps} we state a general link-homotopy
  classification for proper embeddings of disjoint punctured spheres in the $4$-ball, generalizing results of \cite{ABMW,AMW,MY7}.
  \item In Section \ref{sec:Cycles}, we classify Spun links of $3$-compo\-nents up to link-homotopy (Proposition \ref{prop:spun}), 
  and we characterize link-homotopically trivial Spun links of any number of components (Proposition \ref{prop:spun2}).
 \item Section \ref{sec:Kernel} gives an obstruction for a surface-link to be concordant to a ribbon one, which in particular implies that the $2$-dimensional analogue of the braid closure map is not surjective (Corollary \ref{cor:closure}). 
 \item In Section \ref{sec:k-slice}, we characterize classical links with vanishing Milnor invariants of length $\leq 2k$ 
in terms of Milnor invariants of surfaces bounded by the link (Proposition \ref{prop:429}), in connection with the 
\emph{$k$-slice conjecture} proved by Igusa and Orr in \cite{IO}. 
\end{itemize}
We note that our construction recovers previously known extensions of Milnor invariants \cite{ABMW,AMW,MY7}, see Section \ref{sec:previous}.
The relation with Orr's invariants of surface-links \cite{Orr1,Orr2} is discussed in Section \ref{sec:orr}.\\[-0.2cm]

Beyond the construction of invariants, another purpose of this paper is to develop of a general theory of cut-diagrams. We introduce in Sections \ref{sec:Isotopies} and  \ref{sec:SingCo} natural equivalence relations on cut-diagrams, corresponding to isotopy and link-homotopy of surface-links, generated by explicit local moves. This provides a combinatorial framework for studying surfaces in $4$--space. 
The authors believe that this is a promising direction for further research (see e.g. \cite{AMY_kirk} for a concrete example). 
\medskip 

The paper is organized as follows. In Section \ref{sec:kutkutkut}, we introduce cut-diagrams and discuss their connection to surface-links. 
Section \ref{sec:PeripheralSystems} is devoted to the group and the peripheral system
associated to a cut-diagram; in particular we give in Section \ref{sec:periph} the main invariance results, and we state in Theorem \ref{alaChen} a Chen–Milnor type presentation for the nilpotent quotients of this group. 
Section \ref{sec:milnor-numbers-cut} introduces Milnor invariants of cut-diagrams, hence in particular of surface-links, and discusses their properties and relation to existing constructions. 
In Section \ref{sec:SingCo} we explain how to adapt the construction to a singular context, and  derive  link-homotopy invariants. 
Section \ref{sec:appli} presents a range of topological applications illustrating the strength of our invariants. 
Finally, Section \ref{vie} contains the proofs of several technical results, including the key Theorem \ref{alaChen}. 

\subsection*{Definitions and Conventions}
In this paper, everything is stated in the smooth category. All surfaces are compact, ordered and oriented. 

A \emph{surface-link} is a proper smooth embedding of a surface $\Sigma$ in the $4$--ball $B^4$.
Two surface-links $f$ and $g$ of $\Sigma$ are \emph{equivalent} if Im$f$ and Im$g$
are ambient isotopic relative to boundary or, equivalently, if $g$ is
smoothly isotopic to $f\circ \varphi$ for some diffeomorphism $\varphi$ of $\Sigma$. 
By diffeomorphism of $\Sigma$, we always mean an orientation, component and boundary component preserving diffeomorphism from $\Sigma$ to itself.

Two surface-links $f$ and $g$ of $\Sigma$ are \emph{concordant}  if there exists an embedding 
$W: \Sigma\times [0,1]\hookrightarrow B^4\times [0,1]$
such that, for each $i$, the image of the $i$-th component of $\Sigma\times [0,1]$ 
intersects $(B^3\times [0,1])\times \{0\}$, resp. $(B^3\times [0,1])\times \{1\}$, along the $i$-th component of Im$f$, resp. Im$g$, respecting the orientation. 
If $\Sigma$ is not closed, we further require that 
$W\left( \partial\Sigma\times [0,1]\right) = f(\partial \Sigma) \times [0,1]$. 
\\
A closed surface-link is \emph{slice} if it bounds a disjoint union of $3$--dimensional handlebodies in the $5$--ball.
We stress that a null-concordant (closed) surface-link is always slice, but the converse is not necessarily true (see Proposition \ref{prop:anginedepoitrine});  
both notions however coincide in the spherical case of $2$--links. 

A \emph{surface-link map} is a continuous map from $\Sigma$ to $\mathbb{R}^4$ or the $4$--ball $B^4$ with pairwise disjoint images; 
if  $\partial \Sigma\neq \emptyset$, then $\partial \Sigma$ embeds in $S^3=\partial B^4$. 
\emph{Link maps} correspond to the case where $\Sigma$ is a union of spheres. 
Note in particular that a surface-link is a surface-link map. 

Two surface-link maps $f$ and $g$ are \emph{link-homotopic} if 
$f$ is homotopic through surface-link maps to some $f'$ such that
Im$f'$=Im$g$ or, equivalently, if 
$g$ is homotopic through surface-link maps to $f\circ\varphi$, for
some diffeomorphism $\varphi$ of $\Sigma$.\footnote{Note that this
    definition coincides for link maps with the definition usually used
    in the litterature \cite{Kolink,Kirk,ST}. \label{pied}}     
Recall that concordance implies link-homotopy for surface-links, see \cite{BT,BPT}. 
\medskip 

 We will make use of the following notation and conventions: 
 \begin{itemize}
  \item given two elements $a,b$ of some group, their commutator is defined as $[a,b]:=a^{-1}b^{-1}ab$, 
       and the conjugate of $a$ by $b$ is given by $a^b:=b^{-1}ab$;   
  \item given two normal subgroups $G_1$ and $G_2$ of a group
      $G$, we denote by $G_1\cdot G_2$ the normal subgroup of $G$ made of
      products $g_1g_2$ with $g_1\in G_1$ and $g_2\in G_2$;
  \item homotopies of paths should always be understood as fixing the endpoints; 
  \item homology groups are with coefficients in $\mathbb{Z}$, and this will be omitted in our notation. 
\end{itemize}

\begin{acknowledgments}
The first author thanks the IRL PIMS-Europe for its hospitality during the period in which much of the work on this paper was completed.  
The second author was partially supported by the project AlMaRe (ANR-19-CE40-0001-01) of the ANR. 
The third author is supported by a JSPS KAKENHI grant Number 26K06817 for Kakenhi, and by a Waseda University Grant for Special Research Projects (Project number 2026C-078). 
We thank Butian Zhang for pointing out  an orientation mistake in Figure \ref{fig:yooohooo}. 
\end{acknowledgments}

\section{Cut-diagrams}\label{sec:kutkutkut}

In this section we introduce cut-diagrams, which are combinatorial objects that can be thought of as generalized surface-link diagrams. 
Before providing a general definition, we introduce cut-diagrams when they arise from diagrams of surface-links. 

\subsection{An intuitive approach}\label{sec:intuition}

Consider a surface-link diagram, which is a generic projection of a
surface-link on $\R^3\times \{0\}$.
This is  a map $\Sigma\looparrowright \R^3$ for a surface $\Sigma$, 
with lines of transverse double points which may meet at triple points and/or end at branch points. 
As in knot diagrams, the over/under information on the two preimages at each line of double points is encoded 
by cutting off a neighborhood of the lowest preimage 
(see the left-hand side of Figure \ref{fig:yooohooo},  \ref{fig:DtoU_D} and \ref{fig:endpoint} for some examples). 
\begin{figure}
\[
\dessin{3.5cm}{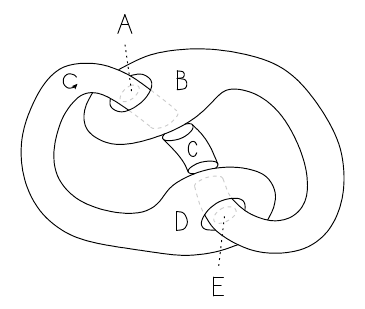}\ \leadsto\ \ \dessin{3.1cm}{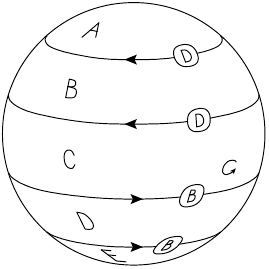}
\]
  \caption{From a surface-link diagram to a $2$--dimensional cut-diagram.
  \\{\footnotesize In all figures, regions are named
    with capital letters, and labels  are given by circled nametags}}
  \label{fig:yooohooo}
\end{figure}
Each line of double points also comes with an orientation, such that
the local frame formed by a positive
normal vector to the overpassing region, a positive
normal vector to the underpassing region, and a positive
tangent vector to the line of double points, agrees with the ambient orientation of $\R^3$.

The lowest preimages of double points form a union $P$ of immersed circles and/or intervals in $\Sigma$, which splits $\Sigma$ into regions. 
Note that the set $P$ is known as the lower deck in the literature (see \cite[\S~4.1]{CS}).
Each triple point of the surface diagram provides an over/under information at the corresponding crossing of $P$, which is encoded as in usual knot diagrams by splitting $P$ into arcs, that we call cut-arcs. 
Each cut-arc inherits an orientation, which is the orientation of the corresponding line of double points, and we label it by the region containing the preimage with highest coordinate of the corresponding line of double points. 
The data of the oriented and labeled cut-arcs in $\Sigma$ forms what we will call a cut-diagram over $\Sigma$ for the given surface-link. 
An example is given in Figure \ref{fig:yooohooo} in the case $\Sigma=S^2$.

We stress that the labeling respects some natural conditions which are inherited from the topological nature of the surface diagram. 
We investigate these conditions below, by analyzing the local cut-diagrams arising from triple and branch points.
\begin{itemize} 
 \item 
An example of local cut-diagram associated with a triple point is illustrated in Figure \ref{fig:DtoU_D}. 
\begin{figure} 
  \[
  \dessin{2.8cm}{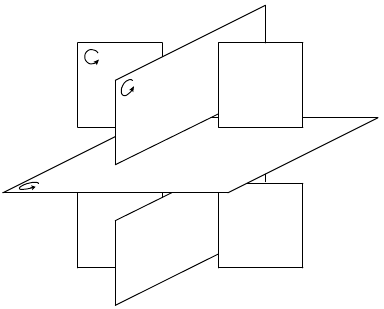}\quad \leadsto\quad \dessin{2.8cm}{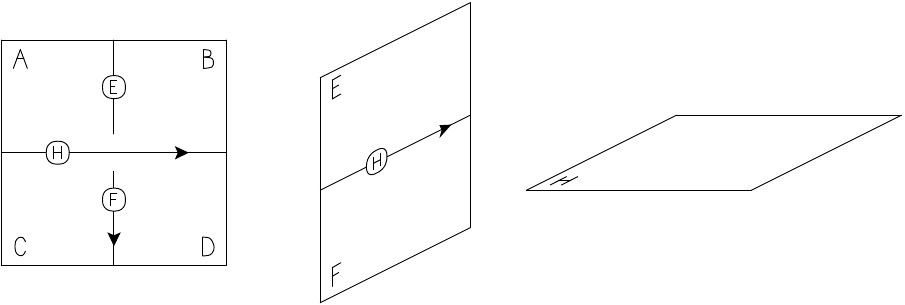}
\]
  \caption{From surface-link diagrams to cut-diagrams: triple point}
  \label{fig:DtoU_D}
\end{figure}
A crossing among cut-arcs occurs in the leftmost sheet of $\Sigma$. There, the overpassing arc is labeled by $H$, whereas the left and right underpassing cut-arcs, with respect to the orientation of the overpassing arc, are labeled by $E$ and $F$,
respectively; these regions $E$ and $F$ are contained in the middle sheet, and are separated by an $H$--labeled cut-arc, with $E$ on the left and $F$ on the right side of this arc.
\item 
Figure \ref{fig:endpoint} shows a local cut-diagram associated to a branch point. 
\begin{figure} 
  \[
\dessin{2.2cm}{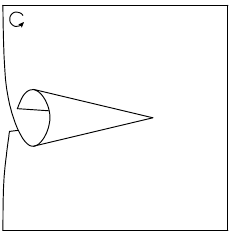}\quad \leadsto\quad \dessin{2.2cm}{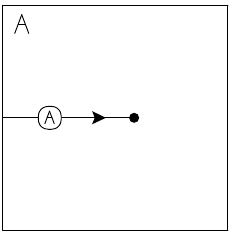}
\]
  \caption{From  surface-link diagrams to cut-diagrams: branch point}
  \label{fig:endpoint}
\end{figure}
This yields a local cut-diagram with a boundary point in the interior of the surface $\Sigma$, 
and the incident cut-arc  is labeled by the region containing this boundary point. 
\end{itemize}
We stress that Figures \ref{fig:DtoU_D} and \ref{fig:endpoint} are only examples of local models, since different cases exist depending on the local plane orientations and over/under informations.

\subsection{Abstract cut-diagrams}\label{sec:abstract}

The above discussion motivates the following definitions, leading to the abstract notion of cut-diagram. 
Let $\Sigma$ be some surface, possibly with boundary.

\begin{defi}\label{def:onlepublierabienunjourcepapier}
A \emph{diagram} over $\Sigma$, is a (non necessarily properly) immersed oriented $1$--manifold in $\Sigma$, 
with only transverse double points, that are endowed with an over/under information as in usual knot diagrams. 

This immersed $1$--manifold splits $\Sigma$ into pieces called \emph{regions}, 
and the over/under decoration at each crossing splits the immersed
manifold into \emph{cut-arcs}, with the induced orientation. 

An \emph{internal endpoint} of a diagram over $\Sigma$ is a point of a boundary component of some cut-arc,  
which is in the interior of $\Sigma$. A cut-arc containing an internal endpoint is called \emph{terminal}.
In figures, internal endpoints will be depicted with a black dot $\bullet$. 
\end{defi}

\begin{defi}\label{def:adjacent}
A \emph{labeling} of a diagram, is an assignment of a region to each cut-arc. 
An ordered pair of regions $(A,B)$ in a labeled diagram is called \emph{$C$--adjacent}, for some region $C$, if 
there exists a smooth path $\gamma:[0,1]\to \Sigma$ such that
$\gamma\big([0,\subfract12[\big)\subset A$, $\gamma\big(]\subfract12,1]\big)\subset B$, 
and $\gamma(\subfract12)$ is a point of a cut-arc labeled by $C$, disjoint from all crossings, 
such that the tangent vector $\gamma'(\subfract12)$ is a positive normal vector for this cut-arc.
We define similarly the notion of \emph{$C$--adjacency} for two cut-arcs which
are separated by a crossing with a $C$--labeled cut-arc.
\end{defi}
\begin{exemple} 
In Figure \ref{fig:DtoU_D},  the pair of regions $(E,F)$ is $H$--adjacent (middle sheet), and 
we observe that likewise, the $E$--labeled and $F$--labeled cut-arcs are $H$--adjacent (leftmost sheet).
\end{exemple}
Inspired by the discussion of the previous section, we set the following: 
\begin{defi}\label{cond}
A labeling is \emph{admissible} 
if it satisfies the following \emph{labeling conditions}:  
\begin{enumerate}
\item[1)] any pair of $A$--adjacent cut-arcs, for some
  region $A$, is labeled by two $A$--adjacent regions; 
\item[2)] a terminal cut-arc containing an internal endpoint
  in some region $A$, is labeled by $A$.
\end{enumerate}
\end{defi}
\noindent These two conditions are locally represented in Figure \ref{fig:cond}.
\begin{figure} 
  \[ \begin{array}{cc}
   \dessin{2.25cm}{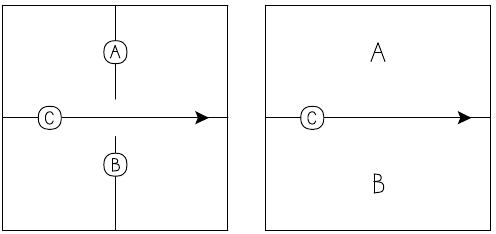}\quad\quad & \quad\quad\dessin{2.25cm}{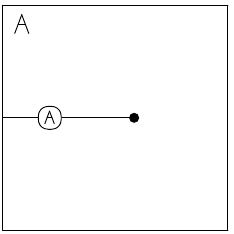}\\
   \textrm{First condition}\quad\quad & \quad\quad\textrm{Second condition}
  \end{array}\]
  \caption{The two labeling conditions}
  \label{fig:cond}
\end{figure}

The left-hand side of Figure \ref{fig:label1} provides an example where the first labeling condition is violated.
There, the two regions $A$ and $B$ are $B$--adjacent but are not $C$--adjacent; 
therefore the crossing of cut-arcs at the center of the figure, where an $A$--labeled and a $B$--labeled cut-arcs are $C$--adjacent, does not respect the first rule. 
\begin{figure} 
  \[ \begin{array}{ccc}
   \dessin{2.25cm}{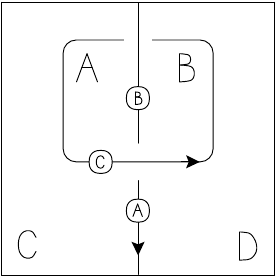} & $\qquad \qquad$ & \dessin{2.25cm}{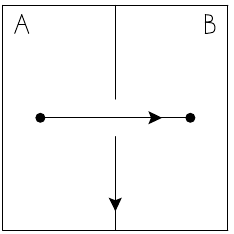}
   \end{array}\]
  \caption{Diagrams that violate the first (left) and second (right) labeling condition }
  \label{fig:label1}
\end{figure}
The right-hand side of Figure \ref{fig:label1} shows a situation where the second labeling condition cannot be satisfied.
Indeed, the left $\bullet$ endpoint sits in region $A$, which imposes that the horizontal cut-arc must be labeled by $A$; 
but this violates the second labeling condition because of the right $\bullet$ endpoint, which lies in region $B$. 

We can finally set the main definition of this section. 
\begin{defi}\label{def:oncerlarondelle}
A \emph{cut-diagram} over $\Sigma$ is a diagram over $\Sigma$ endowed with an admissible labeling. 
\end{defi}

As discussed in the preceding section, there is a canonical way to associate a cut-diagram to a given surface-link diagram, 
and the labeling conditions are automatically satisfied in this setting.
We say that a cut-diagram is \emph{topological} if it arises from a surface-link in this way. 
A number of examples  can be found in Sections~\ref{sec:lair} and \ref{sec:appli}. 
We stress that not all cut-diagrams are topological. 

\subsection{Local moves for cut-diagrams}
\label{sec:Isotopies}

Recall that two surface-link diagrams represent isotopic surface-links if and only if they differ by a sequence of the seven Roseman moves given in \cite{Roseman3}. 
These Roseman moves are straightforwardly translated into moves on cut-diagrams, and it is not hard to check that they are generated by the seven local \emph{cut-moves} C$i$ ($i=0,\cdots,6$) given in Figure \ref{fig:IsotopyDim2}. An alternative, equivalent set of moves is given in \cite[Fig.~3.4]{AMY_kirk}. 

\begin{remarque}\label{rem:labels}
Move C0 may merge two regions or split a region into two: in this case, all cut-arcs  labeled by one of the two merged regions are relabeled by the new region, and all cut-arcs labeled by the split region have to be relabeled in an admissible way\footnotemark 
\, by any of the two new regions. 
All other moves in Figure \ref{fig:IsotopyDim2} are \quote{local}, but may delete or create a region:
such moves are only valid if this region never occurs as the label of any cut-arc. 
\end{remarque}
      \footnotetext{{Note that such an admissible relabeling may not exist, in which
      case the move is not valid; if several different relabelings are
      admissible, then the move exists in several valid versions.}}
      
 \begin{figure} 
\centerline{$
  \begin{array}{ccccc}
    &\hspace{.5cm}&\dessin{1.7cm}{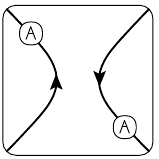} \stackrel{\textrm{C0}}{\longleftrightarrow} 
                   \dessin{1.7cm}{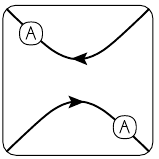}&\hspace{.5cm}&\\[1cm]
    \dessin{1.7cm}{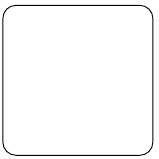} \stackrel{\textrm{C1}}{\longleftrightarrow} 
    \dessin{1.7cm}{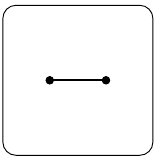}
    &&
       \dessin{1.7cm}{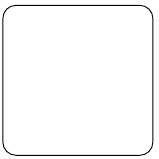} \stackrel{\textrm{C2}}{\longleftrightarrow} 
       \dessin{1.7cm}{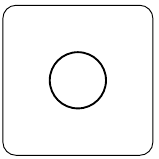}
                 &&
                    \dessin{1.7cm}{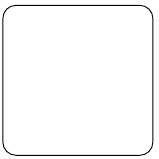} \stackrel{\textrm{C3}}{\longleftrightarrow} 
                    \dessin{1.7cm}{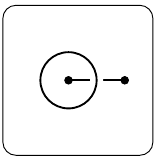}  \\[1cm]
    \dessin{1.7cm}{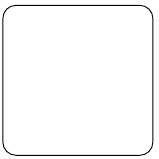} \stackrel{\textrm{C4}}{\longleftrightarrow} 
    \dessin{1.7cm}{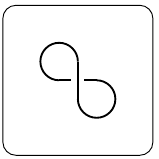}
    &&
       \dessin{1.7cm}{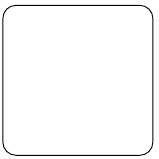} \stackrel{\textrm{C5}}{\longleftrightarrow} 
       \dessin{1.7cm}{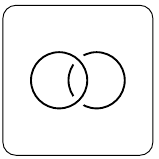}
                 &&
                    \dessin{1.7cm}{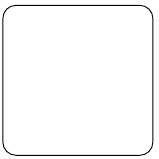} \stackrel{\textrm{C6}}{\longleftrightarrow} 
                    \dessin{1.7cm}{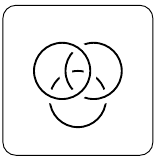}
  \end{array}
$}
  \caption{The cut-moves C$i$ ($i=0,\cdots,6$) on cut-diagrams:\\
  \footnotesize{
we only indicate the labelings when necessary, with the implicit condition that a move is valid only if the labeling conditions are fulfilled before and after the move}
}
\label{fig:IsotopyDim2}
\end{figure}  

Note that two embeddings of a surface having the same image in $4$-space may have different topological cut-diagrams, since two such cut-diagrams might differ by 
some diffeomorphism of $\Sigma$. This motivates taking the following equivalence relation on cut-diagrams. 
\begin{defi}\label{def:cuteq}
 Two cut-diagrams $\UU$ and $\UU'$ over $\Sigma$ are \emph{equivalent} if there exists a 
diffeomorphism $\varphi$ of $\Sigma$  such that $\varphi(\UU)$ and $\UU'$ are related by a sequence of cut-moves. 
\end{defi}

With this definition, we have the following as a consequence of Roseman's theorem \cite[Thm.~1]{Roseman3}.
\begin{prop}\label{cprop}
Two topological cut-diagrams representing isotopic surface-links are equivalent. 
\end{prop}

\subsection{The $1$--dimensional case}\label{sec:1}

We conclude this section by outlining an analogous notion in the context of classical knot theory. 

A  knot diagram is a generic immersion $S^1\looparrowright \R^2$ with a finite number of transverse double points, 
endowed with the usual over/under information on the two preimages. 
Applying the same routine as in the surface-link case, we can consider for each crossing the preimage with lowest coordinate: 
this forms a collection of points on $S^1$, that we call \emph{cut-points},
which split $S^1$ into a number of arcs, that we call \emph{regions}. 
To each cut-point is associated a sign---the sign of the
corresponding crossing, and a region---the overpassing arc at
the crossing. This construction generalizes to diagrams of links and tangles; 
see Figure \ref{fig:1DimCut} for examples.

More generally, one can define abstractly a notion of \emph{$1$--dimensional cut-diagram}, as an oriented $1$--manifold  
together with a collection of cut-points endowed with a sign and labeled by some region. (Note that the natural adaptation of Definition \ref{cond} to this $1$-dimensional setting yields vacuous labeling conditions.)
\begin{figure} 
  \[
    \dessin{2.3cm}{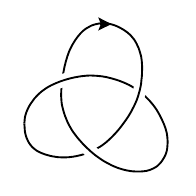} \ \leadsto \, \dessin{2.7cm}{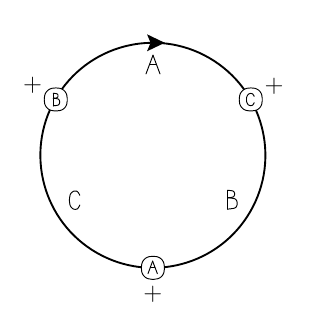}\quad ; \quad
    \dessin{2.5cm}{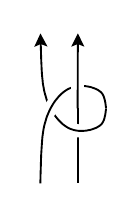} \ \leadsto\, \dessin{2.7cm}{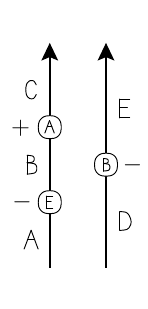}
  \]
  \caption{From diagrams to a $1$--dimensional cut-diagrams.\\{\footnotesize Regions are
      named with capital letters, and labels on cut-points are given by circled nametags}}
  \label{fig:1DimCut}
\end{figure}

This notion is reminiscent of Gauss diagrams \cite{GPV}. 
Indeed, cut-points can be thought of as heads of arrows in a Gauss diagram, where the label represents the region  where the tail is attached. 
Note however that the relative position of tails attached to a same region is not specified in this language. So $1$--dimensional
cut-diagrams should rather be thought of as Gauss diagrams modulo the
local move which exchanges two adjacent tails. 
In fact, by translating Reidemeister moves into this langage of $1$--dimensional cut-diagrams, one recovers the theory of welded knotted objects. 
Hence cut-diagrams can be seen as a $2$--dimensional generalization of Gauss diagrams and/or welded theory. 
This will however not be pursued in this paper, whose main focus is surface-links.  
Consequently, in what follows, the term cut-diagram will be exclusively used for the $2$--dimensional setting of Section \ref{sec:abstract}. 

\section{Group and peripheral systems of cut-diagrams}\label{sec:PeripheralSystems}

One important property of cut-diagrams is that they contain all the relevant information to define 
\quote{fundamental} groups and their distinguished elements, meridians and
longitudes. 

Fix a surface $\Sigma$ with connected components $\Sigma_1, \cdots, \Sigma_n$, and let $\UU$ be a cut-diagram over $\Sigma$.

\subsection{The group of a cut-diagram}

\begin{defi}\label{def:pi1}
The \emph{group $G(\UU)$ of $\UU$} is the group abstractly generated by its regions, and satisfying the Wirtinger relation $B^{-1}A^C$ for each pair $(A,B)$ of $C$--adjacent regions. 
When seen as generators of $G(\UU)$, regions will be called \emph{meridians} of $\UU$, and more precisely \emph{$i$-th meridian} when they belong to $\Sigma_i$. 
\end{defi}
We will see in Proposition \ref{prop:PeripheralSystemIsotopyInvariance} that equivalent cut-diagrams have isomorphic groups. Moreover, this notion is sensible from a topological point of view, as the following illustrates. 
\begin{prop}\label{rem:topopi1}
The group of a topological cut-diagram coincides with the fundamental group of the exterior of the underlying surface-link.
The above notion of meridian then agrees with the usual topological notion of meridian. 
\end{prop}
\begin{proof}
This is a direct consequence of the classical algorithm that gives a presentation for the fundamental group of the exterior of a surface-link from a surface diagram, see e.g. \cite[pp.~88--89]{CKS}. One first assigns generators to the connected components of the diagram: these are in one-to-one correspondence with the cut-diagram regions. Relations are then read off locally from the singular set of the diagram: each curve of double points yields a Wirtinger-type relation among the three generators meeting along this curve, which is the exact same conjugation relation as set in Definition \ref{def:pi1}; no further relation arises from triple or branch points. 
\end{proof}

\begin{exemple}\label{ex:wegottarunrunrunrunrunrunrunrunrunrunrun1}
Consider the $2$-component surface-link $H$ shown on the left-hand side of Figure \ref{fig:Zu-Bead-Ah}.  A topological cut-diagram for $H$ is given on the right-hand side of the figure.
This  will serve as a simple running example throughout this section: see Examples \ref{ex:wegottarunrunrunrunrunrunrunrunrunrunrun2} and \ref{ex:wegottarunrunrunrunrunrunrunrunrunrunrun3}.
\begin{figure} 
  \[
    \dessin{2.8cm}{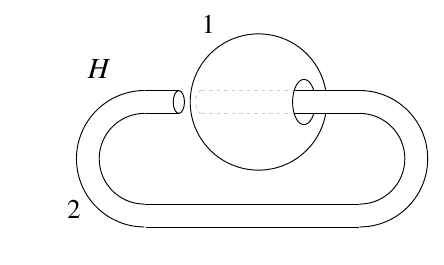} \,\,\, \leadsto\,\,\, \dessin{2.8cm}{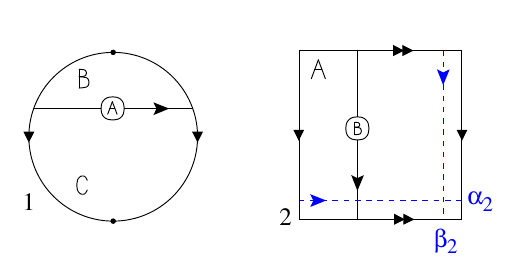}
  \]
  \caption{Diagram for the surface-link $H$ and associated cut-diagram $\UU_{H}$}
  \label{fig:Zu-Bead-Ah}
\end{figure}

The torus-component of $H$ has a single region $A$ and a $B$-labeled cut-arc, while the sphere-component has two regions $B$ and $C$, such that the pair $(B,C)$ is $A$-adjacent.
The group of $H$ thus has presentation
\[ G(H) = \langle A,B,C\, \vert \, A=A^B, \, C=B^A \rangle \simeq \langle A,B\, \vert \, BA=AB\rangle \simeq \mathbb{Z}\oplus \mathbb{Z}. \]
\end{exemple}

\begin{nota}\label{nota:W}
We denote by $\bar{F}$ the free group generated by all meridians of $\UU$, 
and by  $W$ the normal subgroup of $\bar{F}$ generated by all Wirtinger relations, so that we have
$G(\UU)=\fract{\bar{F}}{W}$. 
\end{nota}

In the sequel, we will be mainly interested in the following quotients. 
\begin{defi}
The \emph{lower central series} $\left( G_q\right)_{q\in\mathbb{N}}$ of a group $G$, 
is the descending series of subgroups defined inductively by $G_1:=G$ and $G_{q+1}:=[G,G_q]$. For $q\in\mathbb{N}$, 
the \emph{$q$-th nilpotent quotient} of $G$ is the quotient $N_qG:=\fract{G}{G_q}$ by the $q$-th term of its lower central series.
\end{defi}
For any $q\in\mathbb{N}$, the \emph{$q$-th nilpotent 
group} of $\UU$ is the nilpotent quotient $N_qG(\UU)$ of $G(\UU)$, 
and we have 
\[N_qG(\UU)\cong \fract{\bar{F}}{\bar{F}_q\!\cdot\! W}. \]

\subsection{Generic path and associated group elements}\label{sec:TopCombLemma}

A path
$\gamma:[0,1]\to \Sigma$ is \emph{generic} if
$\gamma(0)$ and $\gamma(1)$ are disjoint from $\UU$,  
and $\gamma\big(]0,1[\big)$ meets $\UU$ transversally in a finite number of regular points of $\UU$, that is, points that are neither double points nor endpoints. 
\begin{lemme}\label{lem:GenMoves}
The following properties hold. 
  \begin{itemize}
  \item Any path on $\Sigma$ with endpoints in $\Sigma\setminus \UU$ is homotopic to a generic path.
  \item If $\gamma_1$ and $\gamma_2$ are two homotopic generic paths, then they are related by a
    finite sequence of homotopies through generic paths, and the
    following local moves which are supported in a $2$--disk in $\Sigma$:
    \begin{itemize}
  \item[] $\textnormal H_1:$ insertion/deletion of a loop going around a regular point of $\UU$\\[-0.3cm]
\[
\dessin{1.5cm}{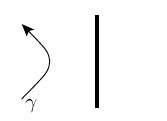}\ \leftrightarrow\ \dessin{1.5cm}{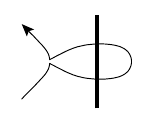};
\]
  \item[] $\textnormal H_2:$ insertion/deletion of a loop going around an internal endpoint\\[-0.3cm]
\[
\dessin{1.5cm}{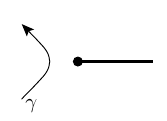}\ \leftrightarrow\ \dessin{1.5cm}{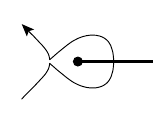};
\]
  \item[] $\textnormal H_3:$ insertion/deletion of a loop going around a double point of $\UU$\\[-0.3cm]
\[
\dessin{1.5cm}{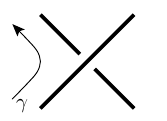}\ \leftrightarrow\ \dessin{1.5cm}{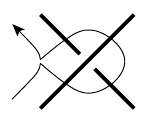}.
\]
  \end{itemize}
  \end{itemize}
\end{lemme}
\begin{proof}
This follows from classical transversality theory, see e.g. \cite[\S~3]{hirsch}. 
Note that $\UU$ provides a stratification of $\Sigma$, where the regular points of $\CC$ form the $1$-dimensional stratum $\Sigma_{(1)}$, while the double points and endpoints of $\UU$ form the $0$-dimensional one $\Sigma_{(0)}$. 

The first point of the lemma is clear: a ($1$--dimensional) path on $\Sigma$ can be perturbed arbitrarily slightly so that it is transverse to each stratum. 
Hence it avoids $\Sigma_{(0)}$ and intersects $\Sigma_{(1)}$ at finitely many points. 

For the second point, consider a homotopy between two paths with fixed endpoints. This homotopy can be perturbed so as to be transverse to the stratification for all but finitely many parameter values. These exceptional parameters correspond to isolated non-generic events, where a path is either tangential to $\Sigma_{(1)}$, or intersects $\Sigma_{(0)}$. The former corresponds to the local move $H_1$, and the latter corresponds to either $H_2$ or $H_3$.
\end{proof}
Given a generic path $\gamma$ on $\Sigma$, we associate unique
words $\widetilde{w}_\gamma $ and $w_\gamma$ of the regions of $\UU$ as follows. 
For the $k$-th intersection point between $\gamma$ and $\UU$ met when running along $\gamma$ according to its orientation,  
denote by $A_k$ the label of the cut-arc met at this point, and by $\varepsilon_k=\pm 1$ the local sign of this intersection point. 
This sign is $1$ if and only if the local orientation given by the orientation of $\UU$ and  $\gamma$ agrees with the ambient orientation of $\Sigma$. 
Then the word $\widetilde{w}_\gamma$ associated to
$\gamma$ is given by 
\[
\widetilde{w}_{\gamma}:=A_1^{\varepsilon_1}\cdots
A_{|\gamma\cap\UU|}^{\varepsilon_{|\gamma\cap\UU|}}, 
\]
whereas the \emph{normalized} word $w_\gamma$  is given by 
\[
w_\gamma:=A^{-|\gamma|} \widetilde{w}_{\gamma},
\]
where $A$ is the region where $\gamma$ starts, and
$|\gamma|$ is the sum of the exponents $\varepsilon_i$ in
$\widetilde{w}_{\gamma}$ such that $A_i$ is in the same connected
component as $\gamma$; see Figure \ref{fig:PathToWord} for an example.

\begin{figure}
  \[
  \dessin{1.5cm}{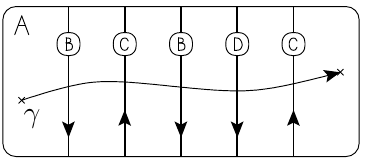}\ \leadsto\
  \begin{array}{l}
    \widetilde w_\gamma=BC^{-1}BDC^{-1}\\[.2cm]
    w_\gamma=ABC^{-1}BDC^{-1}
  \end{array}
  \left(\hspace{-.15cm}\begin{array}{c}
\textrm{\footnotesize assuming $C$ and $D$ are on the same connected}\\
 \textrm{\footnotesize component of $\Sigma$ as $A$, but $B$ is not}
    \end{array}\hspace{-.15cm}\right)
  \]
  \caption{An example of word associated to a path}
  \label{fig:PathToWord}
\end{figure}

Any generic path on $\Sigma$ gives rise in this way to elements $\widetilde{w}_{\gamma}$ and $w_{\gamma}$ in the group $G(\UU)$.

\begin{remarque}\label{rem:ark}
If $\gamma$ is a generic path joining two regions $A$ and $B$, then the Wirtinger
relations imply that $B=A^{\widetilde{w}_\gamma}=A^{w_\gamma}$ in $G(\UU)$. 
This observation has two noteworthy consequences:
\begin{enumerate}
\item any two meridians of a same connected component of $\Sigma$ are conjugate, hence $G(\UU)$ 
  is normally generated by any choice of one meridian on each connected component; 
\item if $\gamma$ is a generic loop based in a region $R$, then $[R,\widetilde{w}_{\gamma}]=[R,w_\gamma]=1$ in $G(\UU)$. 
\end{enumerate}
\end{remarque}

\begin{lemme}\label{lem:Longitudes}
If $\gamma$ and $\gamma'$ are two homotopic generic paths in $\Sigma$,
then $w_\gamma=w_{\gamma'}$ in $G(\UU)$.
\end{lemme}
\begin{proof}
   It is sufficient to check this equality when $\gamma$ and $\gamma'$ differ by a one of the three moves of Lemma \ref{lem:GenMoves}.
  The case of move $\textnormal H_1$ is clear.

  In the case of move $\textnormal H_2$, $\gamma$ splits into $\gamma_1.\gamma_2$, where
  $\gamma_1$ starts at the same region $B$ as $\gamma$, while
  $\gamma_2$ starts at the region $A$ supporting the local move:
\[
\dessin{2cm}{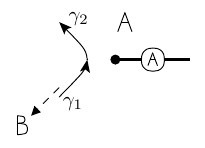}.
\]
  Then we have $w_\gamma=B^{s}\widetilde{w}_{\gamma_1}\widetilde{w}_{\gamma_2}$, while 
  $w_{\gamma'}=B^{s-\e}\widetilde{w}_{\gamma_1}A^\e \widetilde{w}_{\gamma_2}$ for some $s\in \Z$ and some $\e=\pm1$. 
  But, since $\gamma_1$ connects $B$ to $A$,
  we have $A=B^{\widetilde{w}_{\gamma_1}}$ (Remark \ref{rem:ark}) so that, in $G(\UU)$:  
    \[ w_{\gamma'}= B^{s-\e}\widetilde{w}_{\gamma_1}\big(B^\e\big)^{\widetilde{w}_{\gamma_1}}\widetilde{w}_{\gamma_2}=w_\gamma. \]
  
  In the case of move $\textnormal H_3$, we may assume up to some moves $\textnormal H_1$ that the overpassing cut-arc is oriented as shown below. Then $\gamma$ splits as follows into $\gamma_1.\gamma_2$:
  \[
   \dessin{2cm}{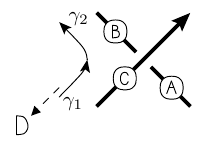}.
  \]
  We have $w_{\gamma}=D^{s}\widetilde{w}_{\gamma_1}\widetilde{w}_{\gamma_2}$ and
  $w_{\gamma'}=D^{s}\widetilde{w}_{\gamma_1}(A^\e)^CB^{-\e}\widetilde{w}_{\gamma_2}$
  for some $s\in \Z$ and some $\e=\pm1$, where $A$, $B$ and $C$ are the labels as shown above, and $D$ is the starting region of $\gamma$. 
  But, by the first labeling condition, the pair $(A,B)$ is $C$--adjacent, 
  hence we have $A^CB^{-1}=1$ in $G(\UU)$.
\end{proof}

\subsection{Longitudes} \label{sec:Longitudes}

We will now define a notion of longitudes, associated to homotopy classes of paths on $\Sigma$. 

For each $i\in\{1,\ldots,n\}$, pick a basepoint $p_i\in\Sigma_i$ in the interior of some region $R_i$. 
This provides in particular a preferred set of normal generators for $G(\UU)$ (Remark \ref{rem:ark}). 
Then, thanks to Lemma \ref{lem:Longitudes}, we can define the \emph{$i$-th loop-longitude map} 
\[
\function{\lambda^l_i}{\pi_1(\Sigma_i,p_i)}
{[\gamma]}{G(\UU)}{w_{\gamma}}, 
\]
where $[\gamma]$ is an element of $\pi_1(\Sigma_i,p_i)$ represented by a generic loop $\gamma$ based at $p_i$. 
Observe that $\lambda^l_i$ is actually a group homomorphism. Indeed, for
any two loops $\gamma_1,\gamma_2$ based at $p_i$, it follows from Remark \ref{rem:ark}~(2) that 
$w_{\gamma_1.\gamma_2}=R_i^{-|\gamma_1|-|\gamma_2|}\widetilde{w}_{\gamma_1}\widetilde{w}_{\gamma_2}=R_i^{-|\gamma_1|} \widetilde{w}_{\gamma_1}R_i^{-|\gamma_2|} \widetilde{w}_{\gamma_2}=w_{\gamma_1}w_{\gamma_2}$. 
\begin{defi}\label{def:long}
The elements in $\Im(\lambda^l_i)$ are called \emph{$i$-th (preferred) loop-longitudes} for $\UU$. 
They depend on the choice of $p_i$, 
hence are only well-defined up to a simultaneous conjugation of $R_i$ and all $i$-th loop-longitudes by some element $w\in G(\UU)$. 
\end{defi}

In general, $\UU$ has an infinite number of loop-longitudes, and it 
will sometimes be useful to consider only a finite generating set of these loops. 
\begin{defi}\label{def:syslong}
For each component $\Sigma_i$ of $\Sigma$, let $\{\gamma_{ij}\}_j$ be a collection of loops based at $p_{i}$ 
which represent a generating set for $\pi_1(\Sigma_i;p_i)$. 
The collection of associated preferred $i$-th loop-longitudes, each denoted by $w_{ij}:= w_{\gamma_{ij}}\in G(\UU)$, is called a \emph{system of loop-longitudes} for $\Sigma$.
\end{defi}
\noindent Observe that any system of  loop-longitudes determines the whole loop-longitude maps $\lambda^l_i$, since these maps are homomorphisms.  

\begin{exemple}\label{ex:wegottarunrunrunrunrunrunrunrunrunrunrun2}
Consider the cut-diagram $\UU_H$ of Figure \ref{fig:Zu-Bead-Ah}. 
The first component being simply connected, the first loop-longitude map $\lambda^l_1$ is trivial. 
For the second component,  a system of loop-longitudes is given by  the loops $(\alpha_2,\beta_2)$ represented in blue dashed lines in the figure. 
Since $\beta_2$ avoids all cut-arcs, we have $\lambda^l_2\left( [\beta_2] \right)=w_{\beta_2}=1\in G(\UU_{H})$. But since $\alpha_2$ intersects the $B$-labeled cut-arc positively, we have $\lambda^l_2\left( [\alpha_2] \right)=w_{\alpha_2}=B\in G(\UU_{H})$. 
\end{exemple}

If $\p \Sigma\neq\emptyset$, we shall also consider \quote{boundary-to-boundary} longitudes. 
For each $i\in\{1,\ldots,n\}$ such that 
$\p\Sigma_i\neq\emptyset$, we fix \emph{(boundary) marked points} $b_{ij}$, $j\in\{0,\ldots,\vert \p\Sigma_i\vert-1\}$ on each boundary component of $\Sigma_i$, and disjoint from $\UU$. 
This provides an ordering of the components of $\p\Sigma_i$, and the one containing $b_{i0}$ is regarded as a \quote{preferred} boundary component of $\Sigma_i$.  
\begin{defi}\label{def:arc-longitudes}
We set $\AA\big(\Sigma_i,\{b_{ij}\}_j\big)$ to be the set of homotopy classes of paths from $b_{i0}$ to some $b_{ij}$ for $j>0$, and we define the \emph{$i$-th arc-longitude map}
\[
\function{\lambda_i^\p}{\AA\big(\Sigma_i,\{b_{ij}\}_j\big)}{[\gamma]}{G(\UU)}{w_\gamma}.
\] 
Here, $[\gamma]$ is the element of $\AA\big(\Sigma_i,\{b_{ij}\}_j\big)$ represented by a generic path $\gamma$ from $b_{i0}$ to some $b_{ij}$. 
Elements of $\Im(\lambda^\p_i)$ are called \emph{$i$-th (preferred) arc-longitudes}. 
\end{defi}
We stress that the well-definedness of arc-longitudes requires the choice of the marked points $\{b_{ij}\}$, and that these are kept fixed at all times. 

\begin{defi}\label{rk:ToBeNamed}
For each component $\Sigma_i$, given a choice of an arc running from $b_{i0}$ to $b_{ij}$ for each $j>0$, the collection of associated preferred $i$-th arc-longitudes is called a \emph{system of arc-longitudes}. \\
Fixing a generic path $\gamma_i$ from $p_i$ to $b_{i0}$, there is a simply transitive left action of $\pi_1(\Sigma_i,p_i)$ on $\AA\big(\Sigma_i,\{b_{ij}\}_j\big)$, defined by $\tau \boldsymbol{\cdot} \gamma:=\tau^{\gamma_i}.\gamma$. 
Hence a system of arc-longitudes is  sufficient to recover the $i$-th arc-longitude map  $\lambda_i^\p$. 
\end{defi}

In the sequel, we shall simply use the terminology \emph{longitude}
when referring to either a  preferred loop- or arc-longitude. 
We denote by 
\[
\lambda_i:\, \pi_1(\Sigma_i,p_i)\sqcup \AA\big(\Sigma_i,\{b_{ij}\}_j\big)\rightarrow G(\UU)
\]
the \emph{$i$-th longitude map} defined by combining $\lambda^l_i$ and $\lambda^\p_i$.

\begin{remarque}\label{rem:topolong}
As far as the authors know, a general notion of (preferred) longitude for surface-links has not been clearly defined in the literature. 
It is however natural to define those as the image, in the fundamental group of the surface complement, 
of either a cycle or a (canonically closed) boundary-to-boundary path on the surface, 
pushed out of the surface in such a way that it has homological intersection zero with the surface.
In this setting, it is clear that the preferred longitudes of a topological cut-diagram agree with the preferred
longitudes of the underlying surface-link. 
\end{remarque}

\subsection{Peripheral system}\label{sec:periph}

The notions introduced so far can now be gathered into the following. 
\begin{defi}\label{def:periph}
A \emph{peripheral system} of $\UU$ is the data $\big(G(\UU);\{R_i\},\{\lambda_i\}\big)$ of the 
group of $\CC$ together with the choice, for each component $\Sigma_i$, of a meridian $R_i$ and
its longitude map $\lambda_i$.\\
Two peripheral systems 
$\big(G(\UU);\{R_i\},\{\lambda_i\}\big)$ and $\big(G(\UU');\{R'_i\},\{\lambda_i'\}\big)$ of cut-diagrams over $\Sigma$
are \emph{equivalent} if there exists a bijection $\rho$ from $\pi_1(\Sigma_i,p_i)\sqcup \AA\big(\Sigma_i,\{b_{ij}\}_j\big)$ to itself which 
is induced by a diffeomorphism of $\Sigma$, 
an isomorphism $\varphi:G(\UU)\rightarrow G(\UU')$ and elements $w_1 , \ldots , w_n\in G(\UU')$   
such that, for every $i\in\{1,\ldots,n\}$, we have $R'_i=\varphi\left(R_i\right)^{w_i}$ and
\[
  \lambda_i'(\gamma)=\left\{\begin{array}{ll}
                             \varphi\left(\lambda_i(\rho(\gamma))\right)^{w_i} & \textrm{for every \emph{loop} $\gamma \in \pi_1(\Sigma_i;p_{i})$,}\\
                             \varphi\left(\lambda_i(\rho(\gamma))\right) & \textrm{for every \emph{arc} $\gamma \in \AA\big(\Sigma_i,\{b_{ij}\}_j\big)$.}
                            \end{array}\right.
                          \]
\end{defi}
\begin{exemple}\label{ex:wegottarunrunrunrunrunrunrunrunrunrunrun3}
For the cut-diagram $\UU_H$ of Figure \ref{fig:Zu-Bead-Ah}, a peripheral system is given by the group $G(\UU_H)\simeq \mathbb{Z}\oplus \mathbb{Z}$ described in Example \ref{ex:wegottarunrunrunrunrunrunrunrunrunrunrun1}, with meridians $\{B,A\}$ (see Figure \ref{fig:Zu-Bead-Ah}), and the longitude maps $\{\lambda_1,\lambda_2\}$. Here the map $\lambda_1$ is trivial since the first component is simply connected; the second component being closed, it admits no nontrivial arc-longitude, so that the map $\lambda_2$ coincides with the loop-longitude map $\lambda^l_2$ of Example \ref{ex:wegottarunrunrunrunrunrunrunrunrunrunrun2}. 
\end{exemple}

\begin{prop}\label{prop:PeripheralSystemIsotopyInvariance}
 Equivalent cut-diagrams have equivalent peripheral systems. 
\end{prop}
\begin{proof}
It suffices to show that (equivalence classes of) peripheral systems for cut-diagrams are invariant under 
the seven cut-moves C$i$ ($i=0,\cdots,6$) given in Figure \ref{fig:IsotopyDim2}. 
We first consider move C0.
Let $\UU$, resp. $\UU'$, be the cut-diagram locally depicted on the left-hand side, resp. right-hand side, of this move in Figure \ref{fig:IsotopyDim2}. 
Denote by $L_1$, $L_2$ and $L_3$ the three depicted regions of $\UU$, from left to right. The three corresponding generators of $G(\UU)$ satisfy the Wirtinger relations $L_2=(L_1)^A$ and $L_2=(L_3)^A$, which in particular implies that $L_1=L_3$ in $G(\UU)$. Turning to $\UU'$, we likewise denote by  $R_1$, $R_2$ and $R_3$ the three depicted regions from top to bottom: these three generators of $G(\UU')$ satisfy $R_1=(R_2)^A$ and $R_3=(R_2)^A$, hence $R_1=R_3$ in $G(\UU')$. 
There is thus an isomorphism from $G(\UU)$ to $G(\UU')$, which is the identity on all regions disjoint from the disk $D$ supporting the move, and which maps $L_1$ to $R_2$ and $L_2$ to $R_1$. \\
It then remains to observe that preferred longitudes are preserved by this isomorphism, since they can always be represented by a path on $\Sigma$ which is disjoint from $D$. 
Some cut-arcs labels may change outside of $D$ under the move (Remark \ref{rem:labels}), but this will not affect the representing words for preferred longitudes thanks to the relations $L_1=L_3$ and $R_1=R_3$. \\[-0.28cm]

  \hspace{-.48cm}
  \parbox[l]{10.6cm}{
   \hspace{.48cm}Let us now consider move C5. 
We denote again by $\UU$ and $\UU'$ the cut-diagrams on the left- and right-hand side of the move, respectively. 
Denote by $R$ the depicted region of $\UU$.
As shown on the right, we also denote by $R$ the \lq outer region\rq \, of $\UU'$, and by $R_1$, $R_2$ and $R_3$ the three new regions. 
Then the presentation for $G(\UU')$ is obtained from that of $G(\UU)$ by adding these three generators, and adding four relations 
  }
  \parbox[r]{1.9cm}{
  \,\,\,\includegraphics[scale=0.85]{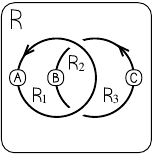}}
\[ \textrm{(1).}\,\,\, R = (R_1)^A \quad ;\quad \textrm{(2).}\,\,\, R_1 = (R_2)^B \quad ;\quad  \textrm{(3).}\,\,\, R = (R_3)^C\quad ;\quad  \textrm{(4).}\,\,\, R_3=(R_2)^A, \]
\noindent where $A$, $B$ and $C$ are the cut-arc labels as shown.\footnote{Note that we made here a counterclockwise choice of orientation, but other cases are completely similar. } 
\\ 
First, we note that (4) follows from the other three relations.
Indeed, on one hand by (3) we have $R_3 = C.R.C^{-1}$ and on the other hand we have
   $$ (R_2)^A = A^{-1}.R_2.A = A^{-1}B.R_1.B^{-1}A = A^{-1}BA.R.A^{-1}B^{-1}A, $$
\noindent where the second, resp. third, equality is given by (2), resp. (1). 
But the first labeling condition of Definition \ref{cond} implies that $(C,B)$ is $A$-adjacent: we thus have the relation $C=B^A$, and relation (4) can be deleted from the presentation of $G(\UU')$. 
\\
We are left with the first three relations, which can be used to delete the three generators $R_2$, $R_3$ and $R_1$ successively by Tietze transformations. 
Preferred longitudes are clearly preserved by the resulting isomorphism $G(\UU)\simeq G(\UU')$, since they can always be made disjoint from the supporting disk.

The verification for the other moves is completely similar to the case of C5; in particular move C6, although involving seven generators and twelve relations, requires no extra idea. 
\end{proof}

It follows in particular from Proposition \ref{cprop} that this
peripheral system is an invariant of surface-links. This is in
agreement with Remark \ref{rem:topolong}, as it corresponds to the topological peripheral system. 
In fact, when restricted to the surface-link case, we have a stronger invariance property, Theorem \ref{th:NilpotentConcordance} below, which uses the following notion. 

\begin{defi}\label{def:icient}
For any $q\ge 1$, the \emph{$q$-th nilpotent peripheral system} of $\UU$ is the data, 
associated to a peripheral system $\big(G(\UU);\{R_i\},\{\lambda_i\}\big)$, 
of the $q$-th nilpotent quotient $N_qG(\UU)$, together with 
the image of each $R_i$ in $N_qG(\UU)$ and the composite of each $i$-th longitude map $\lambda_i$ 
with the projection from $G(\UU)$ to $N_qG(\UU)$. 
This composite is called \emph{$i$-th nilpotent longitude map}. 
Definition \ref{def:periph} naturally induces a notion of equivalence for $q$-th nilpotent peripheral systems.
\end{defi}
It is well-known that equivalence classes of nilpotent peripheral systems 
of links in $S^3$ are concordance invariants \cite{casson}.
This remains true for surface-links.
\begin{theo}
\label{th:NilpotentConcordance}
  Equivalence classes of nilpotent peripheral systems for surface-links are invariant under concordance.
\end{theo}
\begin{proof} 
Let $S_0$ and $S_1$ be two concordant surface-links $\Sigma\hookrightarrow B^4$. 
Let $W$ be a concordance in $B^4\times [0,1]$ between $S_0$ and $S_1$, i.e., 
$W\cap (B^4\times\{\varepsilon\}) = S_\varepsilon\times \{\varepsilon\}~(\varepsilon=0,1)$.
We shall denote the images of $S_0$, $S_1$ and $W$ by the same letter. 

A standard Mayer-Vietoris argument implies that  for $\varepsilon\in\{0,1\}$, the natural inclusion map 
\[B^4\setminus S_\varepsilon\longrightarrow B^4\times [0,1]\setminus W\]
induces isomorphisms $H_k(B^4\setminus S_\varepsilon)\cong H_k(B^4\times [0,1]\setminus W)$ for $k=1,2$.   
More precisely, for $\varepsilon=0,1$ the Mayer-Vietoris sequence of the pair
\[X_\varepsilon=(B^4\times\{\varepsilon\})\cup \overline{U}(W)\quad \textrm{and}\quad Y=B^4\times[0,1]\setminus U(W)\] 
induces isomorphisms $H_k(X_\varepsilon\cap Y)\cong H_k(Y)~(k=1,2)$, 
where $U(W)$ is a tubular neighborhood of $W$ in $B^4\times[0,1]$ and $\overline{U}(W)$ is the closure of $U(W)$.
Since $X_\varepsilon\cap Y$ and $Y$ are homotopic to $B^4\setminus S_\varepsilon$ and 
$(B^4\times [0,1])\setminus W$ respectively, we have the desired isomorphisms. 
It then follows by Stallings Theorem \cite[Thm.~5.1]{Stallings} that the natural inclusion maps induce isomorphisms 
\[\varphi_\varepsilon: N_q(\pi_1(B^4\setminus S_\varepsilon))\stackrel{\cong}{\longrightarrow} N_q(\pi_1(B^4\times[0,1]\setminus W)) \textrm{, for all $q\geq 1$.}\]

Now let $\mathcal{C}_{\varepsilon}$ be a topological cut-diagram for $S_{\varepsilon}$ over $\Sigma$ ($\varepsilon=0,1$), and let 
\[
\lambda^{\varepsilon}_i:\, \pi_1(\Sigma_i,p_i)\sqcup \AA\big(\Sigma_i,\{b_{ij}\}_j\big)\rightarrow 
N_qG(\mathcal{C}_{\varepsilon})
\cong N_q\pi_1(B^4\setminus S_{\varepsilon}) 
\]
be the $i$-th nilpotent longitude map of $\mathcal{C}_{\varepsilon}$ (Definition \ref{def:icient}), where the last isomorphism uses Proposition \ref{rem:topopi1}. 
Since the concordance $W$ induces a diffeomorphism $f_W: S_0\rightarrow S_1$, the surface diffeomorphism $S_1^{-1}\circ f_W \circ S_0:\Sigma\rightarrow \Sigma$
induces a bijection 
\[\rho:\pi_1(\Sigma_i,p_i)\sqcup \AA\big(\Sigma_i,\{b_{ij}\}_j\big)\rightarrow \pi_1(\Sigma_i,p_i)\sqcup \AA\big(\Sigma_i,\{b_{ij}\}_j\big), \]
such that the following diagram commutes:
\[
\vcenter{\hbox{\xymatrix{
 \pi_1(\Sigma_i,p_i)\sqcup \AA\big(\Sigma_i,\{b_{ij}\}_j\big)
 \ar[dd]^(.48){\rho}  \ar[r]^(.48){\lambda^{0}_i}    &  N_qG(\mathcal{C}_{0})\cong N_q\pi_1(B^4\setminus S_{0}) \ar[d]_(.48)\cong^(.48){\varphi_0}  \\
 & 
N_q\pi_1(B^4\times[0,1]\setminus  W) \ar[d]_(.48)\cong^(.48){\varphi_1} \\
  \pi_1(\Sigma_i,p_i)\sqcup \AA\big(\Sigma_i,\{b_{ij}\}_j\big) \ar[r]^(.48){\lambda^{1}_i}    &  N_qG(\mathcal{C}_{1})
\cong N_q\pi_1(B^4\setminus  S_{1})  }}}.
\]
This implies that the nilpotent peripheral systems of $\mathcal{C}_0$ and $\mathcal{C}_1$ are equivalent. 
\end{proof}

\subsection{A Chen--Milnor Theorem for cut-diagrams}
\label{sec:ChenMaps}

In this subsection we give an analogue, for cut-diagrams, of the so-called Chen--Milnor presentation for the nilpotent quotients of fundamental groups of link complements \cite[Thm.~4]{Milnor2}. 

A central tool in establishing the Chen--Milnor presentation for links is an inductively defined sequence of maps $\eta_q$. 
These are not only a central theoretical tool, but they also provide 
an effective algorithm to compute the image of preferred longitudes in
nilpotent quotients, by expressing recursively each $i$-th meridian as a conjugate of a chosen meridian $R_{i}$. 
The major difference with the link case is that, in the context of cut-diagrams, 
there is no canonical way to express these conjugations; as a matter
of fact, we shall fix paths
joining any region to the basepoint of its connected component.

Fix a basepoint $p_i$ on each connected component $\Sigma_i$, away from $\UU$. 
Regions of $\UU$ living in $\Sigma_i$ will be denoted by $R_{ij}$, with $R_i:=R_{i0}$ the region containing $p_i$. 
We also denote by $F=\langle R_i \rangle$  and $\bar{F}=\langle
R_{ij} \rangle$ the free group generated by the $R_i$'s and the
$R_{ij}$'s, respectively.

\begin{defi}
A \emph{road network} $\alpha$  for $\UU$ is the choice, on each component $\Sigma_i$ of
$\Sigma$, of a collection of oriented generic paths $\alpha_{ij}$, called \emph{roads}, running from $p_i$ to a point in each region $R_{ij}$. 
We shall sometimes refer to $\alpha$ as a road network \emph{based at $\{p_i\}$}.
\end{defi}
Following Section \ref{sec:TopCombLemma}, a word $v_{ij}:=\widetilde{w}_{\alpha_{ij}}\in \bar{F}$
can be associated to each road 
$\alpha_{ij}$.
Notice that the relation  $R_{ij}=R_i^{v_{ij}}$ holds in $G(\UU)$ (see Remark \ref{rem:ark}). 

\begin{defi}\label{def:Chen} 
We define \emph{Chen homomorphisms} 
 $\eta_q^\alpha:\bar{F} \rightarrow F$ by
 setting, for all $i,j$ and $q\ge1$:
 \begin{gather*}
   \eta_1^\alpha(R_{ij}):=R_i,\\
\eta_{q+1}^\alpha(R_i):=R_i\ \textrm{ and } \
\eta_{q+1}^\alpha(R_{ij}):=R_i^{\eta_q^\alpha(v_{ij})}.
 \end{gather*}
\end{defi}

As the notation suggests, this sequence of homomorphisms \emph{a priori} depends on the choice of road network $\alpha$. 
However, for simplicity we will often drop the $\alpha$ from the notation.
\begin{remarque}
We use the terminology Chen homomorphisms  
for these maps $\eta_q$ . In the study of links, they indeed first appeared, implicitly, in the work of Chen \cite{Chen}.
 Specifically, in the notation of \cite[\S 4]{Chen}, we have $\eta_q=\varphi\circ \psi^{q-1}$. 
 The inductive definition for $\eta_q$ appeared in \cite{Milnor2}. 
\end{remarque}

Now for each component $\Sigma_i$ of $\Sigma$, pick a system of loop-longitudes
$\mathcal{L}_i(\Sigma):=\{w_{ij}\}_j$ associated with a collection of loops $\{\gamma_{ij}\}_j$ based at $p_{i}$ (see Definition \ref{def:syslong}).
The proof of the following key result is postponed to Section \ref{vie3}. 
\begin{theo}\label{alaChen} 
If $\UU$ is a cut-diagram over $\Sigma$, 
then for each $q\in \mathbb{N}$ we have the following presentation for the $q$-th nilpotent quotient $N_qG(\UU) $ of $G(\UU)$:
\[
\left\langle R_1,\ldots,R_n\ \ \Bigg|
  \begin{array}{l}
    F_q\, ; \, \big[R_i,\eta_q(w_{ij})\big]\textrm{ for all $i$
    and all }\ w_{ij}\in \mathcal{L}_i(\Sigma)
  \end{array}
\right\rangle.
\]
\end{theo}

\begin{remarque}
  Theorem \ref{alaChen} does not only provide a presentation
    for $N_qG(\UU)$ derived from a system of loop-longitudes. It
    also gives, via the Chen homomorphisms, an algorithm to
  compute representative words for any element in $N_qG(\UU)$, in particular for the longitudes in this quotient.
\end{remarque}

\section{Milnor invariants for cut-diagrams and surface-links}
\label{sec:milnor-numbers-cut}

The purpose of this section is to define Milnor invariants for any cut-diagram over a surface $\Sigma$. 
In the classical setting of links in the $3$--sphere, Milnor invariants are  
extracted from the coefficients in the Magnus expansion of the preferred longitudes, seen in the nilpotent quotients of the fundamental group of the link exterior. 
Now, as discussed in Section \ref{sec:PeripheralSystems}, the notion of 
longitudes for cut-diagrams comes in two flavors: loop-longitudes and arc-longitudes. 
As a matter of fact, although coming from a same global construction that associates a collection of \quote{Milnor numbers} 
to any generic path on $\Sigma$ (Section \ref{sec:milnor_path}), 
Milnor invariants defined from loop and arc-longitudes have rather different behaviors, 
and we shall discuss them separately (Sections \ref{sec:muloop} and \ref{sec:muarc}). 
Although we are going to define here invariants for the cut-diagram $\UU$, we emphasize that they are determined by its 
nilpotent peripheral systems, as defined in Section \ref{sec:periph}.  
\medskip 

Let us recall some notation: 
$\Sigma=\sqcup_{i=1}^n\Sigma_i$ is a surface, and $\UU$ is a 
cut-diagram over $\Sigma$. 
For all $i$, we pick a basepoint $p_i$ in the interior of $\Sigma_i$, which specifies an $i$-th meridian $R_i$; we denote by $F:=\langle R_i\rangle$ the free group generated by these meridians.
We also fix, for each $i\in\{1,\ldots,n\}$ such that $\p\Sigma_i\neq\emptyset$, a markedpoint
$b_{ik}$ on each boundary component of $\Sigma_i$ ($k\in\{0,\ldots,\vert \p\Sigma_i\vert-1\}$). 
Finally, we fix a system of loop-longitudes $\mathcal{L}_i(\Sigma)$ for each $i$
(Definition \ref{def:syslong}). 
\subsection{Milnor numbers associated with a longitude}\label{sec:milnor_path}
Fix an integer $q\ge 1$. 
By Theorem \ref{alaChen}, the meridians $R_k$ provide a generating set for $N_q G(\UU)$. 
Given a generic path $\gamma$ in $\pi_1(\Sigma_i,p_i)\sqcup \AA\big(\Sigma_i,\{b_{ij}\}_j\big)$ for some
$i\in\{1,\ldots,n\}$, we can thus consider a representative word $\omega_\gamma$ in $\{R_k^{\pm 1}\}$ for the associated $i$-th longitude in $N_qG(\UU)$.\footnote{This word is not to be confused with the normalized word $w_\gamma$, defined in Section \ref{sec:TopCombLemma}, 
which is a word in the larger alphabet $\{R_{kj}^{\pm 1}\}$.}  
In what follows we shall often blur the distinction between
the path $\gamma$ and the associated $i$-th longitude, seen either in $G(\UU)$ or in $N_qG(\UU)$.

\begin{defi}
 Let $\Z\langle\!\langle n \rangle\!\rangle$ denote the ring of  formal power series in $n$ noncommuting variables $X_1,\ldots,X_n$. 
 The \emph{Magnus expansion}  $E(\omega)\in \Z\langle\!\langle n \rangle\!\rangle$ of an element $\omega\in F$  is 
obtained by substituting each $R_i$ by $1+X_i$, and each $R_i^{-1}$ by $\sum_{k\ge 0} (-1)^k X_i^k$.\\
 For any sequence $I=j_1\cdots j_k$ of integers in $\{1,\ldots,n\}$, we denote by 
 $\mu_\UU(I;\omega)$ the coefficient of the monomial
 $X_{j_1}\!\cdots X_{j_k}$ in $E(\omega)$, meaning that
 \[
E(\omega) = 1+ \sum_{j_1\cdots j_k} \mu_\UU(j_1\cdots j_k;\omega) X_{j_1}\cdots X_{j_k},
\]
where the sum runs over all sequences of (possibly repeated) integers in $\{1,\ldots,n\}$. 
\end{defi}

In particular, we obtain in this way a collection of integers $\mu_\UU(I;\omega_\gamma)$, 
one for each sequence $I$ of integers in $\{1,\ldots,n\}$, from the representative word $\omega_\gamma$. 
Of course, the coefficients $\mu_\UU(I;\omega_\gamma)$  are not determined by the $i$-th longitude $\gamma$, 
as they depend on the choice of the representative word $\omega_\gamma$. 
But Theorem \ref{alaChen} tells us precisely how two different representatives may differ. Indeed we have a presentation 
\begin{equation}
\label{eq:Presentation}
N_q G(\UU) \cong \Big\langle R_i \ \big|\  F_q \,; \,
\big[R_i,\omega_{ij}\big] \textrm{ for all $i,j$} \Big\rangle,
\end{equation}
where, for each $i,j$, the $\omega_{ij}$'s are given word representatives in $\{R_k^{\pm 1}\}$ for 
the elements of the system of $i$-th loop-longitudes $\mathcal{L}_i(\Sigma)$. (We do not need here the explicit formula for $\omega_{ij}$, 
given in Theorem \ref{alaChen} in terms of the Chen homomorphisms.)

In order to extract invariants from the numbers $\mu_\UU(I;\omega_\gamma)$, we introduce the following indeterminacies.
\begin{defi}\label{def:tones}
For any sequence $Ii$ of integers in $\{1,\ldots,n\}$, we define 
\begin{itemize}
\item $m_\UU(Ii):=\gcd\big\{ \mu_\UU(I;\omega_{ij}) \textrm{ for all } j\big\}$; 
\item $\Delta_\UU(I)$ is the greatest common divisor of all  
 $m_\UU(J)$, where $J$ is any sequence obtained from $I$ by deleting at least one index, and permuting the resulting sequence cyclically.
\end{itemize}
\end{defi}
 \noindent We stress that the indeterminacies of Definition \ref{def:tones} only involve \emph{loop}-longitudes on $\Sigma$, and more precisely the chosen systems of loop-longitudes. We also stress that, while the definition of $m_\UU(Ii)$ only uses the system of $i$-th loop-longitudes $\mathcal{L}_i(\Sigma)$, the indeterminacy $\Delta_\UU(Ii)$ involves all systems $\mathcal{L}_j(\Sigma)$ such that index $j$ appears in the sequence $Ii$, since the definition involves cyclic permutation.
\begin{remarque}\label{rem:lesptitcas}
We set $m_\UU(I)=\Delta_\UU(I)=0$ for any sequence $I$ of length $1$ as a convention; this in particular implies that $\Delta_\UU(I)=0$ for any length $2$ sequence $I$.  
\end{remarque}

 \begin{prop}\label{prop:Milnor1}
 For any sequence $I$ of at most $q-1$ integers in $\{1,\ldots,n\}$
 and any $\gamma\in \pi_1(\Sigma_i,p_i)\sqcup
 \AA\big(\Sigma_i,\{b_{ij}\}_j\big)$ for some $i\in\{1,\ldots,n\}$,
 the residue class 
 \[
  \overline{\mu}_\UU(I;\gamma):=\mu_\UU\big(I;\omega_\gamma\big)\mod\Delta_\UU(Ii)
 \]
 does not depend on the representative word $\omega_\gamma$ for $\gamma$. 
\end{prop}
This allows for the following. 
\begin{defi}
  The classes $\overline{\mu}_\UU(I;\gamma)$ are called \emph{Milnor numbers} associated with $\gamma$. 
\end{defi}

\begin{remarque}\label{rem:KKOQQ}
The lower central series being decreasing, any representative word for 
an element of $G(\UU)$ in $N_qG(\UU)$ is also a representative word for its image in $N_{q'}G(\UU)$, for any $q'<q$. 
In particular, by taking $q$ to be sufficiently large, Milnor numbers are defined for
sequences of arbitrary length. 
\end{remarque}

\begin{proof}[Proof of Proposition \ref{prop:Milnor1}]
We begin with a definition that already appears in the work of Milnor in the classical link case, and which will also be a key tool in our context. 
For every $i\in\{1,\ldots,n\}$, we denote by $D_i$ the set of all power series 
$\sum v(i_1\cdots i_k) X_{i_1}\!\cdots X_{i_k}$ in $\Z\langle\!\langle n \rangle\!\rangle$ whose coefficients satisfy 
$v(i_1\cdots i_k)\equiv 0\mod \Delta_\CC(i_1\cdots i_ki)$ for all $i_1\cdots i_k$ with $k<q$. 
\\
Note that, by definition, the Magnus expansions of two words representing an $i$-th longitude yield the same residue classes up to degree $q$ 
if their difference is in $D_i$.
Moreover, following verbatim Milnor's argument, the following statements are straightforwardly checked.\footnote{Points (ii) and (iii) of Claim \ref{claim:Milnor} are not used in the present proof, but will be needed later.}
\begin{claim}\label{claim:Milnor}
  \begin{itemize}
  \item[]
  \item[(i)]  The set $D_i$ is a two-sided ideal of $\Z\langle\!\langle n \rangle\!\rangle$; see
    \cite[(16)]{Milnor2}.
  \item[(ii)]  For any $j\in\{1,\ldots,n\}$ and any representative word $\omega_i\in F$ for an $i$-th loop-longitude, 
    both $X_j\big(E(\omega_i)-1\big)$ and
    $\big(E(\omega_i)-1\big)X_j$ are elements of $D_i$; see
    \cite[(17)]{Milnor2}.
  \item[(iii)]  Let 
    $\mu_\UU(j_1\cdots j_k;\omega_i) X_{j_1}\!\cdots X_{j_k}$ be a term
    in $E(\omega_i)$, for some word $\omega_i\in F$ representing an  
    $i$-th loop-longitude, and let $\iota_j(X_{j_1}\!\cdots X_{j_k})$ be obtained from $X_{j_1}\!\cdots X_{j_k}$ by inserting at
    least one copy of $X_j$,  
    then $\mu_\UU(j_1\cdots j_k;\omega_i) \iota_j(X_{j_1}\!\cdots X_{j_k})$ is
    in $D_i$; see \cite[(18)]{Milnor2}.
 \item[(iv)]  For any $j$ and $k$, both $X_j\big(E(\omega_{jk})-1\big)$ and
    $\big(E(\omega_{jk})-1\big)X_j$ are in $D_i$; see \cite[(19)]{Milnor2}.
  \item[(v)]  For any element $w\in F_q$, $E(w)$ is equivalent to
    $1$ modulo $D_i$; see \cite[(20)]{Milnor2}.
   \end{itemize}
\end{claim}

Now, by (\ref{eq:Presentation}), two word representatives of an element
in $N_qG(\UU)$ differ by a finite sequence of insertions/deletions of 
\begin{enumerate}
\item $R_j^{\pm1}R_j^{\mp1}$ for some $j$;
\item $q$--iterated commutators;
\item commutators $[R_j,\omega_{jk}]$ for some $j$ and $k$.
\end{enumerate}
It is hence sufficient to check that  any such word insertion modifies the Magnus expansion by an element of $D_i$.
Case (1) is clear. Case (2) is a direct consequence of (v) and
(i). Case (3) can be shown, using (iv) and (i), with
the exact same argument as in \cite[pp.~294]{Milnor2}.
\end{proof}

We stress that a choice of one meridian per component has been made for defining $\overline{\mu}_\UU(I;\gamma)$; this will be addressed in the next subsection \ref{sec:bruh}. The definition of $\Delta_\UU(Ii)$ also uses this choice, as well as the choice of a system of loop-longitudes for each component: in the rest of this subsection, we show that $\Delta_\UU(Ii)$ is actually independent of these choices.
We will need the next technical lemma. 
\begin{lemme}\label{lem:Additivity}
 Let $I$ be a sequence of integers in $\{1,\ldots,n\}$, and let 
 $\omega_i\in F$ be a representative word of an 
 $i$-th loop-longitude. Then for any $w_1,w_2\in F$ we have 
   \[
   \mu_\UU(I;w_1\omega_i w_2) \equiv \mu_\UU(I;\omega_i)+\mu_\UU(I;w_1w_2)\mod\Delta_\UU(Ii).
    \]
\end{lemme}
\begin{proof}
  The Magnus expansion being a group homomorphism, we have  
  $$E(w_1\omega_i w_2) =  E(w_1w_2) + \Big(1+\big (E(w_1)-1 \big)\Big)\big( E(\omega_i)-1\big)\Big(1+\big(E(w_2)-1 \big)\Big).$$
  Using (ii) and (i) of Claim \ref{claim:Milnor}, this implies that 
  $E(w_1\omega_i w_2) -  E(\omega_i) - E(w_1w_2)+1$ is in the ideal $D_i$ used in the proof of Proposition \ref{prop:Milnor1},   which gives the conclusion. 
\end{proof}

As a first application of Lemma \ref{lem:Additivity} we have the following, which shows the independence of the indeterminacy $\Delta_\UU(Ii)$ on the choice of systems of loop-longitudes; we also state this result for $\gcd\big\{\Delta_\UU(Ii),m_\UU(Ii)\big\}$ in preparation for the next subsection (see Definition \ref{def:MilnorLOOP}). 
\begin{lemme}\label{lem:4.7}
For any sequence $Ii$ of indices in $\{1,\cdots,n\}$, both $\Delta_\UU(Ii)$ and $\gcd\big\{\Delta_\UU(Ii),m_\UU(Ii)\big\}$ are independent of the choices of systems of loop-longitudes $\mathcal{L}_j(\Sigma)$ ($j\in \{1,\ldots,n\}$) and of representative words  
for the elements of $\mathcal{L}_j(\Sigma)$. 
\end{lemme}
\begin{proof}
The proof is by induction on the length of $Ii$. 
The base case of the induction, where $I$ consists of a single index, is clear from Remark \ref{rem:lesptitcas}, Proposition \ref{prop:Milnor1} and Lemma \ref{lem:Additivity}.
Now, if the result holds for all sequences of length $\le k$, then it holds as well for $\Delta_\UU(Ii)$ for any length $k+1$ sequence $Ii$, just by definition. Proposition \ref{prop:Milnor1} then implies that $m_\UU(Ii)$ does not depend on the choice of  representative words for the elements of $\mathcal{L}_j(\Sigma)$, for all $j$. 
Now suppose that we are given two different systems of $i$-th loop-longitudes $\mathcal{L}_i(\Sigma)$ and $\mathcal{L}'_i(\Sigma)$.
Any element $\omega'$ of $\mathcal{L}'_i(\Sigma)$ decomposes as a product of the elements $\{\omega_{ij}\}$ of $\mathcal{L}_i(\Sigma)$ and their inverses. Lemma \ref{lem:Additivity} can then be used to express $\mu_\UU(I;\omega')$ as a $\mathbb{Z}$-linear combination of all $\mu_\UU(I;\omega_{ij})$'s and $\Delta_\UU(Ii)$ (note in particular that Lemma \ref{lem:Additivity} implies that $\mu_\UU(I;\omega_{ij})+\mu_\UU(I;\omega_{ij}^{-1})\equiv 0 \mod \Delta_\UU(Ii)$). 
This shows that $\mu_\UU(I;\omega')$ is divisible by the gcd of $\Delta_\UU(Ii)$ and all $\mu_\UU(I;\omega_{ij})$'s, and the result follows by symmetry. 
\end{proof}

\subsection{Milnor invariants for cut-diagrams}\label{sec:bruh}

At this point, each generic path, and in particular each longitude of a cut-diagram, yields Milnor numbers. 
We now need to analyse how these numbers provide actual invariants of the cut-diagram. 
On one hand, the choice of meridians $\{ R_k\}$ must be discussed. 
On the other hand, comparing longitudes for two different cut-diagrams is in general not possible, as there is no natural identification between the elements of the fundamental group of the underlying surfaces. 

\subsubsection{Milnor invariants for loop-longitudes}\label{sec:muloop}

We first focus on the case where $\gamma$ is an element of $\pi_1(\Sigma_i;p_i)$ representing an $i$-th loop-longitude.  
By Lemma \ref{lem:Additivity}, the residue class $\overline{\mu}_\UU(I;\gamma)$ remains unchanged when replacing 
$\gamma$ by a conjugate. 
Actually, Lemma \ref{lem:Additivity} implies that, for fixed meridians $\{p_i\}$, Milnor numbers for loop-longitudes 
define group homomorphisms from $\pi_1(\Sigma_i,p_i)$ to $\fract{\Z}{\Delta(Ii)\Z}$. Since the latter is abelian, 
these homomorphisms factor through the abelianization $H_1(\Sigma_i)$.
We have the following. 
\begin{prop}\label{prop:Milnor2}
 Let $I$ be any sequence of integers in $\{1,\ldots,n\}$,  and let 
 $\gamma\in \pi_1(\Sigma_i;p_i)$ for some $i\in\{1,\ldots,n\}$. 
 Then $\Delta_\UU(Ii)$ and $\overline{\mu}_\UU(I;\gamma)$ remain unchanged under a basepoint change on the $j$-th component, 
 for any $j\in\{1,\ldots,n\}$.
 \end{prop}
\begin{proof}
  The proof is by induction on the length of $I$, the initial case being trivial. 
  The statement for the indeterminacy $\Delta_\UU(I)$ is an immediate consequence of the induction hypothesis. 
  
  As just observed, $\overline{\mu}_\UU(I;\gamma)$ only depends on the homology class of $\gamma$. However, we must also consider how changing a meridian $R_j$ modifies a word of $\big\{R_{k}^{\pm 1}\big\}$ representing $\gamma$. 
  More precisely, pick a word $\omega$ of $\big\{R_{k}^{\pm 1}\big\}$ representing $\gamma$, where $R_1,\ldots, R_n$ are the chosen meridians. 
  Let $R'_j$ be the new $j$-th meridian specified by the basepoint change, which  
  is a conjugate of $R_j$, and set $R'_k=R_k$ for all $k\neq j$.  
  A word $\omega'$ of $\{R'_k\}$ representing
  $\gamma$ is thus obtained from $\omega$ by replacing each $R_j$ by a conjugate of $R_j'$.\footnote{See Corollary \ref{rem:baze} for an explicit formula.}  
  Now, a conjugate $(R_j')^g$ of $R'_j$, for some word $g$ of $\{R'_k\}$, is mapped by the Magnus expansion to the product 
  $E(g)^{-1} (1+X_j)E(g)$. Therefore, $E(\omega')$ is obtained from $E(\omega)$ by replacing each occurence of $X_j$ by 
  $$ X_j + (E(g)^{-1}-1)X_j + X_j (E(g)-1) + (E(g)^{-1}-1)X_j(E(g)-1).$$
  It then follows from Claim \ref{claim:Milnor}~(iii) that 
  $E(\omega)-E(\omega')$ is in the ideal $D_i$ introduced in the proof of Proposition \ref{prop:Milnor1}, hence the result.
\end{proof}

Proposition \ref{prop:Milnor2} leads to the following.
\begin{defi}\label{def:MilnorMap}
 For any sequence $Ii$ of integers in $\{1,\cdots,n\}$, we define the  \emph{Milnor map}
   $$ M_\UU^{Ii}:\, H_1(\Sigma_i)\rightarrow \fract{\Z}{\Delta(Ii)\Z}, $$
 by sending any $\lambda\in H_1(\Sigma_i)$ represented by a closed loop $\gamma$ to the class $\overline{\mu}_\UU(I;\gamma)$.  
\end{defi}
Proposition \ref{prop:Milnor2} and Lemma \ref{lem:Additivity} readily give: 
\begin{theo}\label{th:Milnor2}
  The Milnor maps are well-defined homomorphisms. 
\end{theo}
Note that Milnor maps $ M_\UU^{Ii}$ are \emph{not} invariants of the equivalence class of the cut-diagram $\UU$, 
as they are not in general invariant under pre-composition by an automorphism of $H_1(\Sigma_i)$, see Remark \ref{rem:MilnorNilpotentMM} below. 
However, we can define numerical invariants of the cut-diagram $\UU$ from loop-longitudes, which essentially record the images of Milnor map. 
\begin{defi}\label{def:MilnorLOOP}
For any sequence $I$ of integers in $\{1,\ldots,n\}$, we set\footnote{Here we use the usual convention that $\gcd(0,0)=0$.}
$$   \newmu_\UU(Ii):= \gcd\big\{\Delta_\UU(Ii),m_\UU(Ii)\big\} $$
and call it \emph{Milnor loop-invariant} of $\CC$ associated with the sequence $I$ and the $i$-th component.
\end{defi}
\begin{remarque}\label{rem:first}
As in the classical settings, Milnor invariants are ordered by their \emph{length}, which is the number of indices in the indexing sequence $Ii$. 
For a cut-diagram $\CC$, let $k>0$ be the largest integer such that $m_\CC(J)=0$ for any sequence $J$ of less than $k$ indices. 
Then for any sequence $I$ of length $k$ we have $\Delta_\CC(I)=0$, and the 
\emph{first non-vanishing} Milnor (loop-)invariants of $\CC$ are the nontrivial invariants $\newmu_\CC(I)$, 
which are simply given by $\newmu_\CC(I)=m_\CC(I)\in \mathbb{N}$. 
\end{remarque}
We next prove that Milnor loop-invariants are determined by the equivalence class of the nilpotent peripheral system: 
\begin{theo}\label{rem:MilnorNilpotent}
Let $q\ge 2$. If two cut-diagrams $\mathcal{C}$ and $\mathcal{C}'$ have equivalent $q$-th nilpotent peripheral systems, then 
$\nu_{\mathcal{C}'}(Ii)=\nu_{\mathcal{C}}(Ii)$ 
for any sequence $Ii$ of $\le q$ indices in $\{1,\ldots,n\}$.\\
Hence Milnor loop-invariants are well-defined invariants of cut-diagrams up to equivalence. 
\end{theo}
\begin{proof}
Consider two cut-diagrams $\mathcal{C}$ and $\mathcal{C'}$ having equivalent nilpotent peripheral systems $\big(N_qG(\UU); \{R_k\}, \{\lambda_k\}\big)$ and $\big(N_qG(\UU');\{R'_k\},\{\lambda_k'\}\big)$.
Following Definition~\ref{def:periph}, 
there is an isomorphism $\varphi: N_qG(\UU)\rightarrow N_qG(\UU')$ and elements $w_1 , \ldots , w_n$ of $N_qG(\UU')$ such that, for each $k$, we have $R'_k=\varphi\left(R_k\right)^{w_k}$. 
There is also a bijection  $\rho$ from $\pi_1(\Sigma_i,p_i)\sqcup \AA\big(\Sigma_i,\{b_{ij}\}_j\big)$ to itself such that, for any $\gamma\in \pi_1(\Sigma_i,p_i)$, we have 
$\lambda_i'(\gamma)=\varphi\left(\lambda_i(\rho(\gamma))\right)^{w_i}$. 
Summarizing, $\lambda_i'(\gamma)$ is obtained from a conjugate of $\lambda_i(\rho(\gamma))$ by substituting each $R_{k}$ with $(R'_{k})^{w_k^{-1}}$. 
The situation is thus completely similar to the proof of Proposition \ref{prop:Milnor2}: we can likewise use Claim \ref{claim:Milnor}~(iii)
to conclude that $\Delta_{\mathcal{C}'}(Ii)=\Delta_{\mathcal{C}}(Ii)$ and $\overline{\mu}_{\mathcal{C}'}(I,\gamma)\equiv\overline{\mu}_{\mathcal{C}}(I,\rho(\gamma)) 
\mathrm{~mod~}\Delta_{\mathcal{C}'}(Ii)$.
This shows that  $\nu_{\mathcal{C}'}(Ii)=\nu_{\mathcal{C}}(Ii)$. 
The latter part of the statement then follows from Proposition \ref{prop:PeripheralSystemIsotopyInvariance}. 
\end{proof}
\begin{remarque}\label{rem:MilnorNilpotentMM}
The case of Milnor maps is slightly different: indeed, since $M_{\mathcal{C}'}^{Ii}(\gamma)=M_{\mathcal{C}}^{Ii}(\rho(\gamma))$, 
Proposition \ref{prop:Milnor2} gives us that $M_{\mathcal{C}'}^{Ii}=M_{\mathcal{C}}^{Ii}\circ \rho^*$, where 
$\rho^*$ is the automorphism of $H_1(\Sigma_i)$ induced by the bijection $\rho$ in the above proof. 
\end{remarque}

\subsubsection{Milnor invariants from arc-longitudes}\label{sec:muarc}

We now turn to the case where $\gamma$ is an $i$-th arc-longitude, meaning in particular that  $\Sigma_i$ is not closed. 
We actually assume here that \emph{each} component of 
$\Sigma$ has nonempty boundary. Note that this is the case for $2$-string links \cite{AMW} and concordances in $4$--space \cite{MY7}, which we aim at generalizing here, see Section \ref{sec:previous}.

Recall that for $i\in \{1,\cdots,n\}$, we have marked points $\{b_{ij}\}$ on each boundary component of $\Sigma_i$ ($j\in \{0,\ldots,\vert\partial \Sigma_i\vert-1\}$),
among which the point $b_{i0}$ indicates the preferred boundary component. 
This provides a \emph{canonical} choice of $i$-th meridian, materialized by the basepoint $p_i=b_{i0}$, and hence well-defined Milnor
numbers associated to each arc-longitude. In other words, as a corollary of Proposition \ref{prop:Milnor1}, 
Milnor numbers associated to arc-longitudes are well-defined invariants of cut-diagrams over $\Sigma$. 

In general, there are however several, possibly infinitely many different arc-longitu\-des running from $b_{i0}$ to some other marked point $b_{ij}$. 
But using the action of loop-longitudes noted in Definition \ref{rk:ToBeNamed}, all arc-longitudes can be recovered from a single one.
More precisely, for any two arcs $\gamma, \gamma'$ running from $b_{i0}$ to $b_{ij}$, 
there exists some $\tau\in \pi_1(\Sigma_i,b_{i0})$ such that $\gamma'= \tau\cdot\gamma$ (see Definition \ref{rk:ToBeNamed}). 
By Lemma \ref{lem:Additivity}, we then have for any sequence $I$ of integers in $\{1,\ldots,n\}$,
$$ \mu_\UU(I;\gamma') \equiv \mu_\UU(I;\gamma) + \mu_\UU(I;\tau) \mod\Delta_\UU(Ii). $$ 
Since $\tau$ decomposes as a product of elements in a system of loop-longitudes (and their inverses), Lemma \ref{lem:Additivity} further implies that 
$$\mu_\UU(I;\gamma') \equiv \mu_\UU(I;\gamma) \mod \newmu_\UU(Ii), $$ 
where 
$\newmu_\UU(Ii)=\gcd\big\{\Delta_\UU(Ii),m_\UU(Ii)\big\}$ was introduced in Definition \ref{def:MilnorLOOP}.

We can thus define intrinsic Milnor invariants from arc-longitudes, as follows. 
\begin{defi}\label{def:milnorARC}
For any sequence $I$ of integers in $\{1,\ldots,n\}$, we set
$$ \newmu^\p_\UU(I;ij) := \mu_\UU(I;\gamma_{ij})\mod \newmu_\CC(Ii), $$
\noindent where $\gamma_{ij}$ is any path from $b_{i0}$ to $b_{ij}$ ($i\in \{1,\cdots,n\}$, $j\in \{0,\ldots,\vert\partial \Sigma_i\vert-1\}$). 
We call it \emph{Milnor arc-invariant} of $\CC$ associated with the sequence $I$ and the $j$-th boundary component of $\Sigma_i$. 
\end{defi}
\begin{remarque}\label{rem:oulade}
If component $\Sigma_i$ has less than two boundary components, there is no nontrivial $i$-th arc longitude and thus no nontrivial Milnor arc-invariant. 
If  $\Sigma_i$ has exactly two boundary components, we can unambiguously denote by $ \newmu^\p_\UU(I;i):=\newmu^\p_\UU(I;i1)$ 
the Milnor arc-invariants extracted from any path running from the preferred boundary component to the other one.
\end{remarque}
Since the notion of equivalence of peripheral systems fixes the boundary (Definition \ref{def:periph}), the proof of Theorem \ref{rem:MilnorNilpotent} also provides the following.
\begin{theo}\label{rem:MilnorNilpotent2}
If two cut-diagrams $\mathcal{C}$ and $\mathcal{C}'$ have equivalent $q$-th nilpotent peripheral systems, then $\newmu^\p_\UU(I;ij) \equiv \newmu^\p_{\UU'}(I;ij)\mod \newmu_\CC(Ii)$ for any sequence $Ii$ of $\le q$ indices in $\{1,\ldots,n\}$ and any $j\in \{1,\ldots,\vert \partial \Sigma_i\vert-1\}$.\\
Hence Milnor arc-invariants are well-defined invariants of cut-diagrams up to equivalence. 
\end{theo}

\subsection{Milnor invariants of surface-links}\label{sec:lair}

 We call \emph{Milnor invariants} of a cut-diagram $\UU$ the collection of invariants defined in Sections \ref{sec:muloop} and \ref{sec:muarc}, 
namely Milnor loop-invariants $\newmu_\UU(Ii)$ and Milnor arc-invariants $\newmu^\p_\UU(I;ij)$, 
for all indices $i,j$ and all sequences of indices $I$. 

\begin{defi}\label{def:MilnorKnottedManifolds}
Milnor maps $M_{S}^{Ii}$ and  Milnor invariants $\newmu_S(Ii)$ and $\newmu^\p_S(I;ij)$ of a surface-link $S$ are the corresponding invariants for any (topological) cut-diagram associated to $S$.
\end{defi}

By Theorems \ref{rem:MilnorNilpotent} and \ref{rem:MilnorNilpotent2}, Milnor invariants only depend on the nilpotent peripheral systems.
We thus have the following, as a consequence of Theorem \ref{th:NilpotentConcordance}.
\begin{cor}\label{cor:MilnorConcordance}
  Milnor invariants of surface-links are concordance invariants.
\end{cor}

\begin{remarque}\label{pourlesgolmons}
Of course in the case of a $2$--link, there is no nontrivial loop-longitude, hence all Milnor loop-invariants vanish. 
Milnor arc-invariants are likewise undefined, see Remark \ref{rem:oulade}.
More generally, if the $i$-th component of a surface-link $S$ is spherical, then Milnor loop-invariants $\newmu_\UU(Ii)$ vanish for all sequences $I$.
\end{remarque}

Let us give a couple of basic examples; more can be found in Section \ref{sec:appli}.
We address Milnor loop-invariants in Examples \ref{ex1} and \ref{ex2}, then Milnor arc-invariants in Example \ref{ex3} .

\begin{exemple}\label{ex1}
Recall from Examples  \ref{ex:wegottarunrunrunrunrunrunrunrunrunrunrun1},  \ref{ex:wegottarunrunrunrunrunrunrunrunrunrunrun2} and \ref{ex:wegottarunrunrunrunrunrunrunrunrunrunrun3} the $2$-component surface-link $H$ of Figure \ref{fig:Zu-Bead-Ah}.  
Pick $R_1=B$ and $R_2=A$ as meridians. 
Since the first component is spherical, we have $\newmu_{H}(I1)=0$ for any sequence $I$ (Remark \ref{pourlesgolmons}). 
Now, has noted in Examples  \ref{ex:wegottarunrunrunrunrunrunrunrunrunrunrun2} and \ref{ex:wegottarunrunrunrunrunrunrunrunrunrunrun3}, 
a system of loop-longitudes for the second component is given by the loops $(\alpha_2,\beta_2)$ represented in blue dashed lines in the figure. 
We have $E(w_{\beta_2})=E(1)=1$ and  $E(w_{\alpha_2})=E(R_1)=1+X_1$.
It follows that $\omu_{\UU_{H}}(I;\beta_2)=0$ for all sequences $I$, while $\omu_{\UU_{H}}(1;\alpha_2)=1$. 
We thus obtain that $\newmu_{H}(12)=1$. 
\end{exemple}
\begin{remarque}
This number $\omu_{\UU_{H}}(1;\alpha_2)=1$ is merely the absolute value of the \emph{linking number} of the loop $\alpha_2$ with the first component of $H$.  
In general, $\newmu(ij)$ is likewise the gcd of the linking numbers of component $i$ with all loops forming a system of loop-longitudes for component $j$. 
\end{remarque}
\begin{exemple}\label{ex2}
Consider next the diagram of the $3$--component surface-link $B$ shown in Figure~\ref{fig:bead}, along with the associated cut-diagram $\UU_{B}$. 
  \begin{figure}
 \[
 \dessin{4cm}{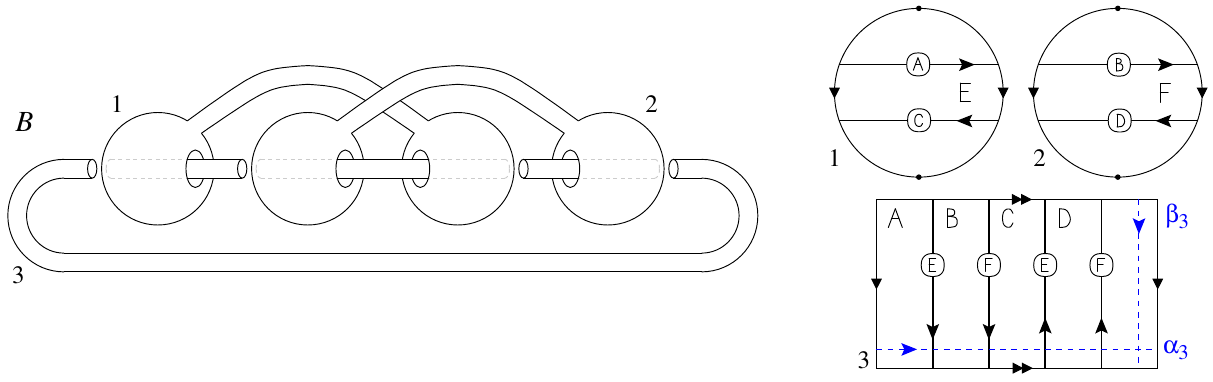}
 \]
  \caption{Diagram of the surface-link $B$ and associated cut-diagram $\UU_{B}$}\label{fig:bead}
  \end{figure}
A system of loop-longitudes for the third component is given by  the loops $(\alpha_3,\beta_3)$ shown in dashed lines in the figure. 
As above, we have $\omu_{\UU_{B}}(I;\beta_3)=0$ for all sequences $I$. 
On the other hand, picking regions $R_1=E$ and $R_2=F$ as meridians for components 1 and 2, we have
\[ E(w_{\alpha_3})=E(R_1R_2R_1^{-1}R_2^{-1}) = 1 + X_1X_2-X_2X_1 + \textrm{ terms of degree $>2$.}
\]
Hence $\newmu_{B}(I)=0$ for any length 2 sequence $I$, and 
 \[
  \newmu_{B}(123)=\newmu_{\UU_{B}}(213)=1\, \, \textrm{ and }\,\,\newmu_{B}(I)=0 \textrm{ for $I\in\{312,132,231,321\}$}. 
 \]
 \end{exemple}
 \begin{remarque}\label{anginedepoitrine}
The surface-links of Examples \ref{ex1} and \ref{ex2} fit in an infinite family of surface-links which is investigated in Section \ref{sec:anginedepoitrine}. 
 \end{remarque} 
 \begin{exemple}\label{ex3}
Consider the diagram of the $3$--component surface-link $M$ shown in Figure \ref{fig:borro}, and   the cut-diagram $\UU_M$ for $M$ shown there. 
 \begin{figure}
\[
\dessin{3.8cm}{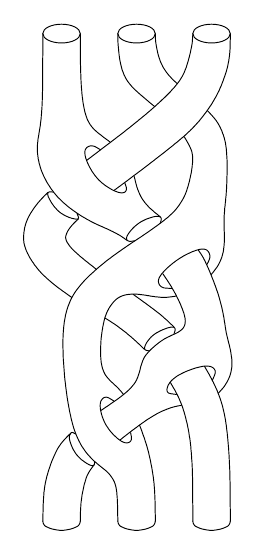}\quad\quad\quad\quad\quad\dessin{3.8cm}{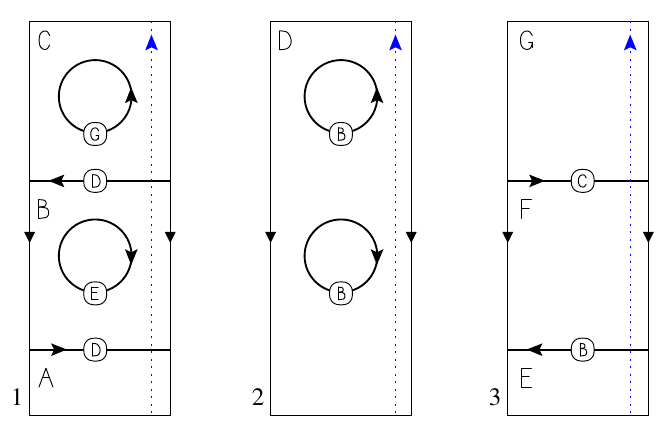}
\]
 \caption{Diagram of the  surface-link $M$ and associated cut-diagram $\UU_M$}\label{fig:borro}
 \end{figure}
For each component of $\UU_M$, we choose the preferred boundary component to be the bottom one, 
so that meridians are given $R_1=A$, $R_2=D$ and $R_3=E$. 

It is easily seen that Milnor loop-invariants $\newmu_{M}(I)$ are zero for any sequence $I$, since
each component admits a system of longitude given by a single loop
running parallel to the lower boundary component  avoiding all cut-arcs. 
Milnor arc-invariants of $B$ are thus well-defined integers. 
A choice of arc-longitude is specified by an arc running from bottom to top on each component. 
Such a choice $\gamma_i$ ($i=1,2,3$) is represented with blue dashed lines in the figure. 
We clearly have $w_{\gamma_1}=DD^{-1}=1$ and $w_{\gamma_2}=1$ in $G(\UU_M)$. 
On the other hand, $w_{\gamma_3}= B^{-1}C = DA^{-1}D^{-1}A$ in $G(\UU_M)$. 
As a result, we obtain (in the notation of Remark \ref{rem:oulade}): 
 \[
  \newmu^\p_M(12;3)=-\newmu^\p_M(21;3)=1\, \, \textrm{ and }\,\,\newmu^\p_M(I;i)=0 \textrm{ for $i=1,2$ and any sequence $I$}. 
 \]
\end{exemple}

\subsection{Relation to previous works}\label{sec:previous}

We now discuss the relationship between our generalized Minor invariants with several related works.

\subsubsection{Milnor invariants of $2$-string links and concordances}\label{sec:previous2}

Milnor arc-invariants generalize the invariants defined in \cite{MY7,AMW,ABMW}. 
These earlier works indeed correspond to the case where $\Sigma$ is a union of annuli properly embedded in $B^4$. 
Note that in this setting, loop-longitudes are given by parallel copies of the boundary components,
so that both $\Delta_\UU(I)$ and $\newmu_\CC(I)$ involve (classical) Milnor
invariants of the link at the boundary. See \cite{MY7} for a complete description (note that in \cite{MY7}, the indeterminacy 
$\newmu_\CC(Ii) $ of Definition \ref{def:milnorARC} is denoted by $\overline{\Delta}(Ii)$).

\begin{remarque}\label{rem:Stallings0}
A key ingredient in \cite{MY7,AMW} is a result of Stallings \cite[Thm.~5.1]{Stallings} that provides the Chen-Milnor presentation for the nilpotent quotients of the fundamental group of the complement. 
In this paper, we use instead the notion of cut-diagram to derive Theorem \ref{alaChen}, see Section \ref{vie}.
This combinatorial method not only  offers more flexibility, but also provide an algorithm for expressing longitudes as words in the meridians,  
hence for computing Milnor invariants. 
\end{remarque}

\subsubsection{Orr invariants}\label{sec:orr}

Orr introduced in \cite{Orr1,Orr2} (based) concordance invariants $\theta_q^l$ for (based) codimension $2$ embeddings ($q\in \mathbb{N}$). 
The construction in the surface-link case can be roughly summarized as follows. 

Let $L$ be an embedding of an $n$--component surface $\Sigma$ into $S^{4}$, and let $G(L)$ be the fundamental group of its complement. 
Suppose that the nilpotent quotients satisfy $\fract{G(L)}{G(L)_q}\cong \fract{F}{F_q}$, where $F$ is the free group of rank $n$.
 
The natural homomorphisms $F\longrightarrow \fract{F}{F_q}$ and $G(L)\longrightarrow \fract{G(L)}{G(L)_q}\cong \fract{F}{F_q}$ induce continuous maps
\[f_q^n:K(F,1)\longrightarrow K\left(\fract{F}{F_q},1\right)\text{~~and~~}
\phi_q^n:S^{4}\setminus L\longrightarrow K\left(\fract{F}{F_q},1\right)\]
respectively, where $K(G,1)$ denotes an Eilenberg-MacLane space for a group $G$.
The map $\phi_q^n$ extends to a map $\overline{\phi_q^n}$ 
from $S^{4}$ to $\mathrm{Cone}(f_q^n)$, the mapping cone of the map $f_q^n$.
The fundamental class $\theta_q^n\in\pi_{4}\left(\mathrm{Cone}(f_q)\right)$ 
of $\overline{\phi_q^l}$ is an invariant of $L$ for each $q\ge 1$, and Orr shows that these $\theta_q^l$ actually are (based) concordance invariants.

The definition of Orr's invariants thus requires that $\fract{G(L)}{G(L)_q}\cong \fract{F}{F_q}$. 
Observe that this condition is fulfilled if $H_1(\Sigma)=0$, which is the (stronger) assumption taken in \cite{Orr1} for the definition. 
We note that $\fract{G(L)}{G(L)_q}\cong \fract{F}{F_q}$ if and only if all Milnor loop-invariants of $L$ with length $<q$ vanish. 
Hence Orr invariants actually are invariants for surface-links with vanishing Milnor loop-invariants.

Moreover, as discussed in  Remark \ref{rem:orr2}, 
we expect that Orr's construction can be adapted to create new link-homotopy invariants of surface-links.
\section{Non-repeated Milnor invariants and link-homotopy}\label{sec:SingCo}
Cut-diagrams were, for a large part, motivated by the study of surface-links and their Milnor \emph{concordance} invariants. 
But Milnor's work was first introduced in the context of \emph{link-homotopy} \cite{Milnor}, and it is natural to extend the theory of cut-diagrams to this setting. 

\subsection{Self-singular cut-diagrams}\label{sec:sylady}
Recall that a \emph{surface-link map} is generically an immersion, with only a finite number of singularities which are transverse double points. 
In diagrams, a singularity corresponds to the local model given on the left-hand side of Figure \ref{fig:FromSingBrokToSingCutDiag}. 
\begin{figure}
   \[
\dessin{3cm}{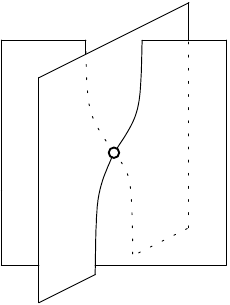}\quad \leadsto\quad \dessin{3cm}{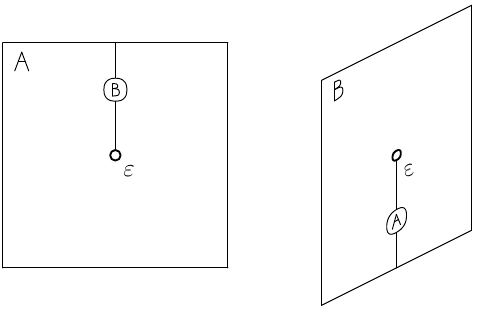}
\]
  \caption{From singular surface-link diagrams to  cut-diagrams ($\varepsilon=\pm$)}
  \label{fig:FromSingBrokToSingCutDiag}
\end{figure}
The procedure explained in Section \ref{sec:intuition}, that associates a cut-diagram to a surface-link diagram, extends naturally to surface-link maps. 
As Figure \ref{fig:FromSingBrokToSingCutDiag} illustrates, each singularity gives rise to a pair of internal endpoints, represented by a white dot $\circ$. 
Note also from the figure that a terminal cut-arc incident to a $\circ$ endpoint, is labeled by the region containing the \emph{paired} $\circ$ endpoint, which belongs to the \emph{same} connected component of the surface. This leads to the following general definition.
\begin{defi}\label{def:self-cut}
  A \emph{self-singular cut-diagram} over  $\Sigma$ is a diagram over $\Sigma$ whose set of internal endpoints is partitioned into $\bullet$ and $\circ$ endpoints, 
  and with a labeling of all cut-arcs by regions, satisfying the first labeling condition of Definition \ref{cond} and the following conditions on terminal cut-arcs: 
 \begin{itemize}
  \item[2)] a terminal cut-arc containing a $\bullet$ endpoint in some region $A$, is labeled by $A$;
  \item[2')] for each connected component of $\Sigma$, the set of
    $\circ$ endpoints is endowed with a partition into $2$-element subsets. 
    Each subset $\{p_1,p_2\}$ is such that, if $p_1$ lies in some region $X$ and has incident cut-arc labeled by $Y$,
    then $p_2$ lies in $Y$ with incident cut-arc labeled by $X$, and the local orientations at $p_1$ and $p_2$ are the same. 
  \end{itemize}
 \end{defi}
 \noindent In figures, we encode the local orientation at a $\circ$ endpoint by a sign, which is positive if and only if the incident cut-arc is locally oriented outwards.\footnote{When the self-singular cut-diagram is obtained from a diagram of a surface-link map, this is precisely the sign of the corresponding singularity. }
 We stress that the pairing of $\circ$ endpoints will not be used in this paper, but is crucial in \cite{AMY_kirk}, where self-singular cut-diagrams are used to define link-homotopy invariants of link maps. 
\medskip

Now, Roseman moves have been extended to the case of surface-link maps in \cite{AMW}, where three self-singular Roseman moves were introduced. 
It is an easy exercise to translate the three self-singular Roseman moves of \cite{AMW} into local moves on self-singular cut-diagrams (see \cite[Fig.~3.5]{AMY_kirk}), 
and to show that they are generated by the four local moves SC$i$ ($i=1,\cdots,4$)  
of Figure \ref{fig:singmoves2}. 
\begin{figure}
\centerline{$
  \begin{array}{ccc}
    \dessin{1.7cm}{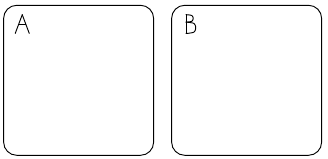} \stackrel{\textrm{SC1}}{\longleftrightarrow} 
     \dessin{1.7cm}{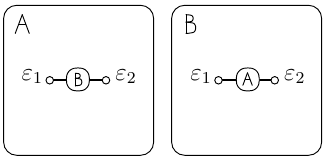}&\hspace{.5cm}&\dessin{1.7cm}{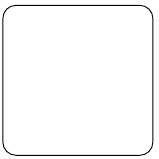} \stackrel{\textrm{SC2}}{\longleftrightarrow} 
                                          \dessin{1.7cm}{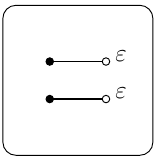}\\[1cm]
     \dessin{1.7cm}{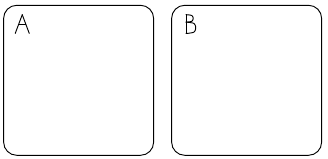} \stackrel{\textrm{SC3}}{\longleftrightarrow} 
    \dessin{1.7cm}{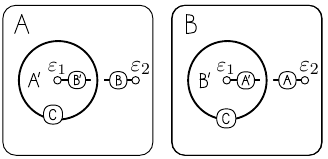}&&\dessin{1.7cm}{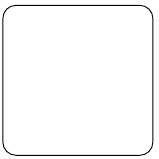} \stackrel{\textrm{SC4}}{\longleftrightarrow} 
    \dessin{1.7cm}{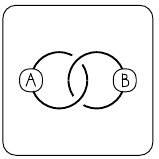}
  \end{array}
  $}
  \caption{The moves SC$i$ ($i=1,\cdots,4$) for self-singular cut-diagrams:\\{\footnotesize
      $A$ and $B$ are always regions of a same connected component; $\circ$ endpoints decorated by the same $\varepsilon$--sign are paired in the sense of Definition \ref{def:self-cut}~b).
      }}\label{fig:singmoves2}
\end{figure} 
\noindent As a consequence of \cite[Prop.~2.4]{AMW}, we have:
\begin{prop}\label{ctresprop}
 Two cut-diagrams of link-homotopic surface-links are related by a sequence of moves SC$i$ ($i=1,\cdots,4$) and equivalences of cut-diagrams.
\end{prop}
\begin{remarque}
A similar statement appears in \cite[Thm.~3.12]{AMY_kirk}. The difference lies in the definition of link-homotopy for surface-link maps, see Footnote \ref{pied}. However, both situations coincide for link maps, which is the main focus of \cite{AMY_kirk}.
\end{remarque}

\subsection{Reduced peripheral system}\label{sec:loug}

Let $\UU$ be a self-singular cut-diagram over $\Sigma$. 
Since internal endpoints play no role in its definition, the group $G(\UU)$ can be defined for $\UU$ in the same way as in Definition \ref{def:pi1}.  
However, in this case, longitudes are no longer well-defined elements in $G(\UU)$, and we consider the following quotient, first introduced by Milnor \cite{Milnor}.  
\begin{defi}\label{def:reduc}
The \emph{reduced quotient} $\nR G(\UU)$ of $G(\UU)$ is its quotient by the normal subgroup generated by all relations $[R,R^g]$, for all meridians $R$ and all $g\in G(\UU)$. 
\end{defi}

\begin{lemme}\label{lem:ReducedLongitudes}
If $\gamma$ and $\gamma'$ are two homotopic generic paths in $\Sigma_i$, for some $i\in\{1,\ldots,n\}$,
then $w_\gamma=w_{\gamma'}$ in $\fract{\nR G(\UU)}{N_{(i)}}$, where $N_{(i)}$ is the normal subgroup generated by $i$-th meridians. 
\end{lemme}
\begin{proof}
  The proof follows the exact same lines as the proof of Lemma \ref{lem:Longitudes}.
  The only difference concerns move $\textnormal H_2$: if the endpoint involved in the move is a $\circ$ endpoint, the terminal cut-arc may not be labeled by the region it is adjacent to, but by any
  conjugate of it. But in such a case, we actually have
  $w_{\gamma'}=\widetilde{w}_{\gamma'}=\widetilde{w}_\gamma=w_\gamma$ in
  $\fract{G(\UU)}{N_{(i)}}$, where we still denote by $N_{(i)}$ the normal subgroup of $G(\UU)$ generated by $i$-th meridians. The result follows, since $\fract{\nR G(\UU)}{N_{(i)}}$ is a quotient of $\fract{G(\UU)}{N_{(i)}}$.
\end{proof}
Now, for each component $\Sigma_i$ of $\Sigma$, let $R_i$ be a choice of meridian, specified by a basepoint $p_i$. 
For each $i\in\{1,\ldots,n\}$, we have by Lemma \ref{lem:ReducedLongitudes} a well-defined map 
\[
\lambda_{i,R}:\, \pi_1(\Sigma_i,p_i)\sqcup \AA\big(\Sigma_i,\{b_{ij}\}_j\big)\rightarrow \fract{\nR G(\UU)}{N_{(i)}}
\]
defined by composing the $i$-th longitude map $\lambda_i$ of Section \ref{sec:periph} with the
projection map from $G(\UU)$ to $\fract{\nR G(\UU)}{N_{(i)}}$. 
Elements in the image of $\lambda_{i,R}$ are called \emph{reduced $i$-th longitudes}. 
These are cosets of the form $\lambda\cdot N_{(i)}$, where
$\lambda$ is a word representing an $i$-th longitude.
\begin{defi}
The data $\big(RG(\UU);\{R_i\},\{\lambda_{i,R}\}\big)$ is called a
\emph{reduced peripheral system} of $\UU$.
\end{defi}
\noindent Definition \ref{def:periph} naturally induces a notion of equivalence for reduced peripheral systems. 

As noted in Remark \ref{rem:ark}, $G(\UU)$ is normally generated by the meridians $\{R_i\}$. 
It follows, by a result of Habegger and Lin \cite[Lem. 1.3]{HL}, that its reduced quotient is
nilpotent of order at most $n$, so that $\nR G(\UU)\cong N_q\nR G(\UU)$ for any integer $q\ge n$. 

As in Section \ref{sec:ChenMaps}, pick for each component $\Sigma_i$ of $\Sigma$, a 
system of loop-longitudes $\mathcal{L}_i(\Sigma)=\{w_{ij}\}_j$ in $G(\UU)$, and recall that $F:=\langle R_i\rangle$. 
We have the following Chen--Milnor type presentation for $\nR G(\UU)$, similar to Theorem \ref{alaChen}. 
\begin{theo}\label{alaChenSing} 
For every $q\geq n$, the reduced quotient of $G(\UU)$ has the following presentation:
\[
\textrm{R} G(\UU) \cong \left\langle R_1,\ldots,R_n\ \ \left|
  \begin{array}{l}
    \big[R_i,R_i^g\big]\textrm{ for all $i$, and all $g\in F$}\\[.1cm]
    \big[R_i,\eta_q(w_{ij})\big]\textrm{ for all $i$ and all $w_{ij}\in \mathcal{L}_i(\Sigma)$} 
  \end{array}
\right.\right\rangle.
\]
\end{theo}
The proof is an adaptation of that of Theorem \ref{alaChen}, and is given in Section \ref{vie4}. 

We next prove the following \lq reduced version\rq\, of Proposition \ref{prop:PeripheralSystemIsotopyInvariance}. 
\begin{prop}\label{prop:ReducedPeripheralSystemHomotopyInvariance}
The equivalence class of the reduced peripheral system is invariant under all moves SC$i$, $i=1,\cdots,4$ and equivalences of cut-diagrams.
\end{prop}
\begin{proof}[Sketch of proof]
Since the reduced peripheral system is determined by the peripheral system, invariance under equivalences of cut-diagrams follows from Proposition \ref{prop:PeripheralSystemIsotopyInvariance}.\\
Next, we outline the proof for move SC4; the other moves are treated in a similar way. 
Let $\UU$ and $\UU'$ be the cut-diagrams on the left- and right-hand side of the move, respectively. 
As in the proof of move C5 (Proposition \ref{prop:PeripheralSystemIsotopyInvariance}), denote by $R$ both the depicted region of $\UU$ and the \lq outer region\rq \, of $\UU'$, and by let $R_1$, $R_2$ and $R_3$ be the three new regions, from left to right.
The presentation for $G(\UU')$ is obtained from that of $G(\UU)$ by adding generators $R_1$, $R_2$ and $R_3$, and the four relations 
\[ \textrm{(1).}\,\,\, R = (R_1)^A \quad ;\quad \textrm{(2).}\,\,\, R_1 = (R_2)^B \quad ;\quad  \textrm{(3).}\,\,\, R = (R_3)^B\quad ;\quad  \textrm{(4).}\,\,\, R_3=(R_2)^A, \]
\noindent where $A$ and $B$ are two regions of the \emph{same} component of $\Sigma$, so that $AB=BA$ in the reduced quotient (Remark \ref{rem:ark}~(i)).  
Combining (3) and (1) gives $R_3= BA^{-1}.R_1.AB^{-1}$, while (2) gives $(R_2)^A=A^{-1}B.R_1.B^{-1}A$, showing that (4) follows from the three other relations in $\textrm{R}G(\UU')$. These three remaining relations, along with the three new generators, can then be deleted by Tietze transformations. 
The resulting isomorphism $\textrm{R}G(\UU)\simeq \textrm{R}G(\UU')$ preserves reduced longitudes, 
since these can always be made disjoint from a disk supporting the move.
\end{proof}

\subsection{Non-repeated Milnor invariants and link-homotopy invariance}
\label{sec:NonRepeatedMilnor}
We have the following \lq reduced version\rq\, of Proposition \ref{prop:Milnor1}. 
\begin{prop}\label{prop:ReadFromReduced}
  Let $\UU$ be a self-singular cut-diagram over $\Sigma$, and $i\in \{1,\ldots,n\}$. 
  Let $I$ be a sequence of pairwise distinct integers in
  $\{1,\ldots,n\}\setminus\{i\}$, and $\gamma$ be some path in $\pi_1(\Sigma_i)\sqcup\AA\big(\Sigma_i,\{b_{ij}\}_j\big)$. 
  Then the residue class 
  \[
   \overline{\mu}_\UU(I;\gamma):= \mu_\CC(I;\omega) \mod \Delta_\UU(Ii),
  \]
does not depend on the choice of word $\omega$ such that $\omega\cdot N_{(i)}$ is the reduced $i$-th longitude associated to $\gamma$.
\end{prop}
This allows for the following, which in particular defines non-repeated Milnor invariants of surface-links and surface-link maps. 
\begin{defi}\label{def:jenaimarre}
We call \emph{non-repeated Milnor invariants} of a self-singular cut-diagram $\UU$, all Milnor loop-invariants $\newmu_\UU(Ii)$ and Milnor arc-invariants $\newmu^\p_\UU(I;ij)$, for indices $i,j$ and sequences $I$ such that all indices in $Ii$ are pairwise distinct.
\end{defi}
\begin{proof}[Proof of Proposition \ref{prop:ReadFromReduced}]
By Theorem \ref{alaChenSing} and the definition of reduced longitudes, two choices of representative words 
are related by a finite sequence of insertions/deletions of
\begin{enumerate}
\item $R_j^{\pm1}R_j^{\mp1}$ for some $j$;
\item commutators $[R_j,\eta_q(w_{jk})]$ for some $j$ and $k$;
\item commutators $[R_j,R_j^g]$ for some $j$ and $g\in F$;
\item conjugates $R_i^g$ for some $g\in F$.
\end{enumerate}
It has already been checked in the proof of Proposition
\ref{prop:Milnor1} that the residue class of $\mu_\UU(I;\omega)$ is
invariant under the operations (1) and (2) (there, we denoted $\eta_q(w_{ij})$ by $\omega_{ij}$). 
For operation (3),  
suppose that two words $\omega$ and $\omega'$ differ by a factor $[R_j,R_j^g]$, for some index $j$ and some $g\in F$. 
A standard argument on the Magnus expansion shows that $E\left([R_j,R_j^g]\right)-1$ is a sum of monomials, 
each containing two occurrences of the variable $X_j$. 
This readily implies that $\mu_\UU(I;\omega)=\mu_\UU\big(I;\omega'\big)$ for any sequence $I$ of pairwise distinct integers. 
Concerning operation (4), since the sequence $I$ does not contain $i$, it is easily seen 
that $\mu_\UU(I;\omega)=\mu_\UU\big(I;\rho_{R_i}(\omega)\big)$, where the operator $\rho_{R_i}$ removes all occurrences of $R_i$. 
If $\omega'$ is obtained from $\omega$ by inserting a $R_i^g$ factor, then $\rho_{R_i}(\omega')$ is obtained
from $\rho_{R_i}(\omega)$ by inserting $\rho_{R_i}(gg^{-1})$, and the result then follows from case (1). 
\end{proof}

By arguments similar to those proving Theorems \ref{rem:MilnorNilpotent} and \ref{rem:MilnorNilpotent2}, we deduce that two self-singular cut-diagrams with equivalent reduced peripheral systems, have same non-repeated Milnor invariants. More precisely: 
\begin{theo}\label{ptittrucenplus}
If two self-singular cut-diagrams $\UU$ and $\UU'$ over $\Sigma$ have equivalent reduced peripheral systems, then the following hold for any non-repeated sequence $Ii$.
\begin{itemize}
  \item $\nu_{\mathcal{C}'}(Ii)=\nu_{\mathcal{C}}(Ii)$; 
  \item If $\Sigma$ has no closed component, then $\newmu^\p_\UU(I;ij) \equiv \newmu^\p_{\UU'}(I;ij)\mod \newmu_\CC(Ii)$. 
\end{itemize}
\end{theo}
Combining Theorem \ref{ptittrucenplus} with Proposition \ref{prop:ReducedPeripheralSystemHomotopyInvariance}, we obtain that non-repeated Milnor invariants of self-singular cut-diagrams are invariant under the moves SC$i$, $i=1,\cdots,4$. 
In particular, we have the following.
\begin{cor}\label{cor:lhinvariance}
Non-repeated Milnor invariants of surface-link maps are invariant under link-homotopy. 
\end{cor}
\begin{remarque}\label{rem:orr2}
Let $L$ be an $n$--component surface-link given by an embedding of $\Sigma$ into $S^{4}$. 
One can adapt the definition of Orr's invariant \cite{Orr1} reviewed in Section \ref{sec:orr}, 
but replacing the nilpotent quotient $\fract{F}{F_q}$ with the reduced free group $\nR F$. 
Under the assumption that $\nR G(L)\cong \nR F$, we obtain in this way \lq new' invariants $\theta_R^n
=[\overline{\phi_R^n}]\in\pi_{4}(\mathrm{Cone}(f_R^n))$ of $L$,
where 
\[ f_R^n:K(F,1)\longrightarrow K(\nR F,1)\text{~~, and~~}
\overline{\phi_R^n}:S^{4}\longrightarrow \mathrm{Cone}(f_R^l)\]
is the extension of $\phi_R^n:S^{4}\setminus L\longrightarrow K(\nR F,1)$.\\
Although this will not be further developed in the present paper, we expect that these invariants $\theta_R^l$ are (based) link-homotopy invariants.
Since the above assumption that $\nR G(L)\cong \nR F$ is equivalent to saying that all non-repeated Milnor invariants of $L$ vanish, 
we could obtain in this way  \lq new\rq\, link-homotopy invariants for surface-links with vanishing non-repeated Milnor invariants. 
This naturally raises the question whether all surface-links with vanishing non-repeated Milnor invariants are link-homotopically trivial; 
this is known to be the case for links in the $3$-sphere \cite{Milnor2}.
\end{remarque}

\section{Examples and applications}\label{sec:appli}

In this section, we gather several properties and topological applications of Milnor invariants of surface-links. 

\subsection{Realization result, and sliceness vs. null-concordance}\label{sec:anginedepoitrine}

We first give a general realization result for Milnor loop-invariants, 
which also highlights the subtle difference between sliceness and null-concordance for surface-links. 

Pick two positive integers  $m$ and $n$ with $n \geq 2$. 
Let $O_{n-1}$ be a copy of the $(n-1)$-component trivial $2$-link in the $4$-sphere $S^4$.
The fundamental group of the complement of $O_{n-1}$ is a free group generated by $x_1,\ldots ,x_{n-1}$, where each $x_i$ is represented by a meridian of the $i$-th component.
Consider a loop $\gamma_n^m$ in the complement of $O_{n-1}$ representing the element
\[ [x_1,[x_2,[x_3,\cdots,[x_{n-2},x_{n-1}]\cdots]]]^m \]
\noindent in the free group. 
Here, we may assume that $\gamma_n^m$ lies in a $3$-sphere in $S^4\setminus O_{n-1}$. 
Let $T_n^m$ be the boundary of a tubular neighborhood of $\gamma_n^m$ in this $3$-sphere.
We consider the surface-link $B_n^m$ defined by 
 \[ B_n^m:=O_{n-1}\cup T_n^m, \]
consisting of a single torus and $n-1$ spheres. 
Note that the surface-links $H$ (Figure \ref{fig:Zu-Bead-Ah}) and $B$ (Figure \ref{fig:bead}) coincide with $B_2^1$ and $B_3^1$, respectively. The following actually generalizes the first computations given in Examples  \ref{ex1} and \ref{ex2}.

\begin{prop}\label{prop:anginedepoitrine}
For any pair of positive integers $m$ and $n$ with $n \geq 2$, the $n$-component surface-link $B_n^m$ satisfies $\newmu_{B_n^m}(I)=0$ for any sequence $I$ of length at most $n-1$, and 
\[ \newmu_{B_n^m}(12\cdots n)=m.\] 
Consequently, $B_n^m$ and $B_n^{m'}$ are neither concordant nor link-homotopic, for any $m \ne m'$. \\
In particular, although $B_n^m$ is slice, it is neither null-concordant nor link-homotopically trivial.
\end{prop}
\begin{proof}
Since the first $n-1$ components of $B_n^m$ are spherical, we have $\newmu_{H}(Ii)=0$ for any sequence $I$ and any $i\in \{1,\ldots,n-1\}$ (Remark \ref{pourlesgolmons}). 
A system of loop-longitudes for the $n$-th component is given by two loops $(\alpha_n^m,\beta_n^m)$, similar to the loops $(\alpha_3,\beta_3)$ shown in Figure \ref{fig:bead} for $B=B_3^1$: the loop $\alpha_n^m$ is a $0$-framed parallel copy of the loop $\gamma_n^m$, and $\beta_n^m$ is a null-homotopic meridian of $\gamma_n^m$. 
We thus have $w_{\beta_n^m}=1$, while $w_{\alpha_n^m}=[x_1,[x_2,[x_3,\cdots,[x_{n-2},x_{n-1}]\cdots]]]^m$. The Magnus expansion of the latter is thus given by 
\[ E(w_{\alpha_n^m}) = 1 + m X_1X_2\cdots X_{n-1} + \textrm{ terms of degree $\ge n-1$.}
\]
The desired values for Milnor loop-invariants of $B_n^m$ follow. 

The fact that $B_n^m$ is a slice surface-link is clear by construction. In fact, it is a ribbon surface-link (see Section \ref{sec:Kernel} for a definition), which readily implies sliceness. The rest of the statement follows from Corollaries \ref{cor:MilnorConcordance} and \ref{cor:lhinvariance}.
\end{proof}
\begin{remarque}
We also note that the surface-link $B_n^m$ is Brunnian for all $m,n$ ($n\ge 2$), meaning that removing any component yields a trivial surface-link. 
\end{remarque}

\subsection{Milnor invariants of Spun links}\label{sec:spun}

Spun links refer to a classical construction due to Artin \cite{Artin}, which produces surface-links from classical tangles as follows. 
Consider in $\R^4$ the upper $3$--dimensional space $\R^3_+=\{(x,y,z,0) \, \vert\, x,y\in\R,\, z\ge 0\}$;  
the $(x,y)$--plane $P_{xy}=\{(x,y,0,0)\, \vert\, x,y\in\R\}$ sits as the boundary of $\R^3_+$. 
Given a compact $1$--manifold
$T$ properly embedded in $\R^3_+$, the \emph{Spun of $T$} is the surface-link $\Spun(T)$ obtained by spinning $T$ around $P_{xy}$ inside $\R^4\supset\R^3_+$: 
\begin{equation}\label{eq:spun}
 \Spun(T) = \left\{\big(x,y,z\cos \theta,z\sin \theta \big)\, \vert \, (x,y,z,0)\in T,\, \theta\in [0,2\pi] \right\}.
\end{equation}
 Observe that a closed component of $T$ yields a toroidal component of
 $\Spun(T)$, while a knotted arc in $T$ produces a spherical
 component.

 Using a similar spinning process in $\R^3$ around the line
 $\{(x,0,0)\, \vert\, x\in\R\}$, a diagram $D$ of $T$ provides a
 natural surface-link diagram of $\Spun(T)$, where each crossing
 yields a circle of double points without triple or branch point. And
 in turn, spinning the $1$--dimensional cut-diagram $\C_D$ associated to $D$ (in the sense of Section \ref{sec:1}) 
 provides a $2$--dimensional cut-diagram $\C_S$ for $\Spun(T)$, where each
 ($1$--dimensional) region of $\C_D$ yields a ($2$--dimensional) region of $\C_S$, and each 
 cut-point of $\C_D$ yields a closed cut-arc of $\C_S$ with same label and oriented
 according to the sign of the cut-point. 
 In particular, the toroidal
 components of $\C_S$ are nothing but the circle components of $\C_D$
 times $S^1$, while the spherical components of $\C_S$ are given by 
 the interval components of $\C_D$ times $S^1$ 
 (see Figure \ref{fig:Lm} for an example.)

Our Milnor invariants are well-behaved under the Spun construction, in the sense that they relate naturally to the (classical) Milnor invariants of $T$: 
\begin{lemme}\label{lem:Milnorspun}
Let $T$ be a tangle in $\R^3_+$ as above. 

 (1)\,If the $i$-th component of $T$ is a knot, then for any sequence $I$ we have 
$$ \newmu_{\Spun(T)}(Ii) = \textrm{gcd}\left\{\mu_T(Ii), \Delta_T(Ii)\right\}. $$

 (2)\,If the $i$-th component of $T$ is an arc, then for any sequence $J$ we have 
$$ \newmu_{\Spun(T)}(Ji)=0. $$ 
\end{lemme}

\begin{remarque}\label{rem:realfirst}
When the $i$-th component of $T$ is a knot, we have in particular that the first non-vanishing Milnor loop-invariants of $\Spun(T)$ are given by the first non-vanishing Milnor invariants of $T$ as
$\newmu_{\Spun(T)}(Ii)=\vert \mu_T(Ii)\vert$.
\end{remarque}

\begin{proof}
Suppose that the $i$-th component of $T$ is a knot, so that the $i$-th component of $\Spun(T)$ is a torus.
Pick a diagram $D$ for $T$ and let $\UU_D$ be the
associated cut-diagram for $T$, with $i$-th component denoted by $\sigma_i$.
Let $\CC_S$ be the corresponding cut-diagram for $\Spun(T)$, as described above, with $i$-th component denoted by $\Sigma_i$.
A system of $i$-th loop-longitudes for $\Sigma_i$ is provided by a choice $(\alpha_i,\beta_i)$ of 
two loops given by 
\begin{equation}\label{eq:spun_basis}
 \alpha_i=\sigma_i\times \{\ast\}\quad\textrm{ and }\quad \beta_i=\{\star\}\times S^1
\end{equation}
for  some point $\ast$ on $S^1$ and  some point $\star$ on $\sigma_i$ which is not a cut-point. 
By construction, the loop $\beta_i$ is disjoint from all cut-arcs in $\Sigma_i$, hence we have that $\mu_{\CC_S}(I;\beta_i)=0$ for any sequence $I$, and 
$\newmu_{\Spun(T)}(Ii) = \textrm{gcd}\left\{\mu_{\CC_S}(Ii), \Delta_{\CC_S}(Ii)\right\}$. 

\noindent From the discussion preceding the statement of the lemma, 
it is easily seen that the Wirtinger presentation given by the diagram $D$ of $T$ 
coincides with that of $G(\CC_S)$, and that the word representing the $i$-th preferred longitude of $T$ 
coincides with $\lambda_i(\alpha_i)$ 
in $G(\CC_S)$.  
This shows that $\Delta_{\UU_S}(Ii)=\Delta_T(Ii)$ and 
$\mu_{\CC_S}(I;\alpha_i) = \mu_T(Ii)$, and part (1) follows. 

Part (2) of the statement is clear: if the $i$-th component of $T$ is an arc, then the $i$-th component of $\Spun(T)$ is a sphere and there is no nontrivial $i$-th loop-longitude (Remark \ref{pourlesgolmons}).
\end{proof}

\begin{remarque}\label{sec:real}
The Spun map can also be used to give another realization result for Milnor loop-invariants. While Section \ref{sec:anginedepoitrine} involves surface-links of a torus with several $2$-spheres, the following involves surface-links of tori. 
Consider the $(n+1)$--component link $M_{n+1}$ shown below. 
\[
 \dessin{2.15cm}{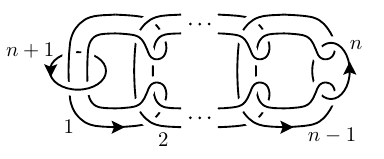}
\]
\noindent Milnor observed in \cite{Milnor} that all Milnor link invariants of length $\leq n$ vanish for $M_{n+1}$, and that 
for any permutation $\sigma$ in $S_{n-1}$, we have  
\vspace{.2cm}
$\,\, \mu_{M_{n+1}}(\sigma(1) \cdots \sigma(n-1)n\ n+1) = \left\{\begin{array}{cc}
1 & \textrm{if $\sigma=$Id,} \\
0 & \textrm{otherwise.}
\end{array}\right.$

\noindent Using Lemma \ref{lem:Milnorspun} and Remark \ref{rem:realfirst}, 
we directly obtain a similar realization result for the first non-vanishing Milnor loop-invariants $\newmu(1\cdots n+1)$ of surface-links, 
by considering $\Spun(M_{n+1})$. 
\end{remarque}

\subsection{A classification result up to concordance} 
\label{sec:Saito}

We now compare the relative strength of our Milnor invariants with other concordance invariants of surface-links in the literature. 

  For every $m\in\N$, we define $W_m$ as the spun surface-link obtained
  by spinning the tangle $X_m$ described in Figure \ref{fig:Lm}.
\begin{figure}
  \[
  W_m:=\Spun\left(\dessin{3cm}{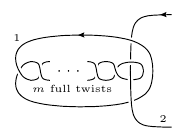}\right)\qquad \qquad \dessin{3.5cm}{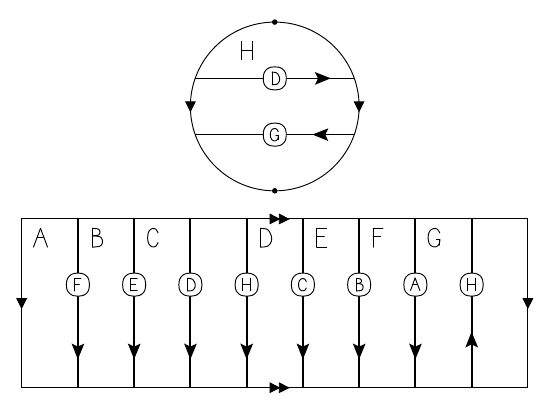}
  \]
  \caption{Definition of the surface-links $W_m=\Spun(X_m)$, and a cut-diagram for $W_3$}
  \label{fig:Lm}
\end{figure}
Many of the known concordance invariants of surface-links cannot detect this family of Spun links.  
Indeed, for all $m\in \N$, one can check that 
\begin{itemize}
 \item the Sato--Levine invariant \cite{Sato} vanishes on $W_m$,
 \item Cochran's derivation invariants \cite{Cochran} all vanish on $W_m$, 
 \item Saito's invariants \cite{Saito} of $W_m$ are equal for all values of $m\in\N$, 
 \item since $W_m$ is link-homotopic to a union of trivially embedded torus and sphere for all $m\in\N$, 
all link-homotopy invariants vanish on $W_m$.
\end{itemize}
In contrast, using Milnor loop-invariants we show the following. 

\begin{prop}\label{prop:Saito}
  For any $m_1,m_2\in\N$, $W_{m_1}$ and $W_{m_2}$ are concordant if and only if $m_1=m_2$.
\end{prop}
\begin{proof}
This is an immediate consequence of the fact that, for all $m\in \N$, $\newmu_{W_m}(2121)=m$. \\ 
This computation can be performed directly on a cut-diagram, as given on the right-hand side of Figure \ref{fig:Lm}  for $m=3$; 
although it is technically slightly more involved, this computation follows the exact same lines as Example \ref{ex2}. 
Alternatively, since $W_m$ is a Spun link, one can also use Lemma \ref{lem:Milnorspun} for this computation. 
Closing the second component of $X_m$ yields a $2$--component link 
which is the Whitehead $m$--double of the negative Hopf link, in the sense of \cite{MY_osaka}. 
Using \cite[Thm.~1.1]{MY_osaka}, we obtain that $\mu_{X_m}(I1)=0$ for any sequence $I$ of length $<3$, and that $\mu_{X_m}(2211)=-m$.
By Lemma \ref{lem:Milnorspun}, we deduce that  $\newmu_{W_m}(2211)=m$, as desired.
\end{proof}

\subsection{Link-homotopy classification results} 
\label{sec:PuncturedSpheres}

The next two applications provide link-homotopy classification results for Spun links and for knotted punctured spheres. 

\subsubsection{Spun links up to link-homotopy}
\label{sec:Cycles}

Milnor showed in \cite{Milnor} that links of $3$ components are classified up to link-homotopy by non-repeated Milnor invariants of length $\le 3$.  
Using the good behavior of our Milnor invariants under the Spun construction, we show the following. 
\begin{prop}\label{prop:spun}
  Let $L$ and $L'$ be $3$--component links. 
  Then $\Spun(L)$ and $\Spun(L')$ are link-homotopic if and only if $L$ is link-homotopic to either $L'$ or its mirror image.
\end{prop}
\begin{proof}
  The \quote{if} part of the statement is clear. 
  Indeed spinning a link-homotopy between $L$ and $L'$ provides a
  link-homotopy between $\Spun(L)$ and $\Spun(L')$, and  it is a well-known fact that the Spun of a link and of its mirror image are isotopic. 
  
  Next we suppose that $F=\Spun(L)$ and $F'=\Spun(L')$ are link-homotopic.
  Let $D$ (resp. $D'$) be a diagram for $L$ (resp. $L'$), and let $\UU$ (resp. $\UU'$) be the topological cut-diagrams for
  $F$ (resp. $F'$) associated to the spinning of $D$ (resp. $D'$), which is a cut-diagram over a disjoint union $\Sigma$ of three tori. 
  As seen in Subsection \ref{sec:loug},  $\UU$ and $\UU'$ have equivalent reduced peripheral systems; this in particular provides a diffeomorphism $\varphi$ from $\Sigma$ to itself (see Definition \ref{def:periph}).
  
  Let $\alpha_i,\beta_i$ (resp. $\alpha'_i,\beta'_i$) be loops on the $i$-th component of $\Sigma$, 
  corresponding to $D_i\times\{*\}$ and $\{\star\}\times S^1$ (resp. $D'_i\times\{*\}$ and $\{\star\}\times S^1$), respectively; here $D_i$ (resp. $D'_i$) denotes the $i$-th component of $D$ (resp. $D'$). 
  Denote by $\gamma_i$ the loop $\varphi(\alpha'_i)$, 
  and let $m_i$ and $n_i$ be integers such that 
  $\gamma_i$ is homologous to $m_i \alpha_i+n_i\beta_i$. 
  
  Let $Ii$ be any non-repeated sequence of indices in $\{1,2,3\}$. 
Using the correspondance between longitudes of $L$ (resp. $L'$)
and those of $\UU$ (resp. $\UU'$) outlined in the proof of Lemma \ref{lem:Milnorspun},
we observe that $\Delta_\UU(Ii)=\Delta_L(Ii)$ and $\mu_\UU(I;\alpha_i)\equiv \mu_{L}(Ii) \mod \Delta_{L}(Ii)$, and likewise for $\UU'$ and $L'$.
  Now, since $\beta_i$ avoids all cut-arcs, we have $\mu_\UU(I;\beta_i)=0$. Hence by Lemma \ref{lem:Additivity},  we obtain: 
    $$ \mu_\UU(I;\gamma_i)\equiv m_i \mu_\UU(I;\alpha_i)\equiv m_i\mu_{L}(Ii) \mod \Delta_{L}(Ii). $$
   Moreover, because $F$ and $F'$ are link-homotopic, we have that 
   $\Delta_\UU(Ii)=\Delta_{\UU'}(Ii)$ (Corollary \ref{cor:lhinvariance}), and the equivalence of reduced peripheral system gives  
     \[\mu_{\UU'}(I;\alpha'_i) \equiv \mu_\UU(I;\gamma_i) \mod \Delta_{\UU}(Ii).\]
  All together, we obtain in this way that  
     \begin{equation}\label{eq:late}
      m_i \mu_{L}(Ii)\equiv \mu_{L'}(Ii) \mod \Delta_{L}(Ii).
     \end{equation}
   On the other hand, using again Corollary \ref{cor:lhinvariance} and (the proof of) Lemma \ref{lem:Milnorspun}, we have that 
     \begin{equation}\label{taface}
         \vert \mu_{L}(Ii)\vert \equiv \vert \mu_{L'}(Ii)\vert \mod \Delta_{L}(Ii).  
     \end{equation}
     
     Let us now focus on length $2$ Milnor invariants of $L$ and $L'$, that is, on linking numbers.
     Pick two indices $i,j$ in $\{123\}$. 
     Since $\Delta(ij)=0$ and $\mu(ij)=\mu(ji)$, 
     Equation (\ref{eq:late}) gives 
     $$ m_i \mu_L(ji)=\mu_{L'}(ji)=\mu_{L'}(ij)=m_j\mu_L(ij), $$
     showing that either $\mu_L(ij)=0$ or $m_i=m_j$. 
     Equation (\ref{taface}) then implies that we either have $\mu_L(ij)=\mu_{L'}(ij)=0$ 
     or $m_i=m_j=\pm 1$. Note that in both cases we have $\Delta_L(123)=\Delta_{L'}(123)$. 
     
     We now turn to length $3$ Milnor invariants. Recall that $\mu(ijk)\equiv \mu(jki)\equiv -\mu(ikj) \mod \Delta(ijk)$ for any indices $i,j,k$, see \cite{Milnor2}. There are two cases: 
     \begin{itemize}
     \item If $\mu_L(ij)=\mu_{L'}(ij)=0$ for all $i,j$, then $\Delta_L(123)=0$ and Equation (\ref{eq:late}) gives  
     $$ m_3 \mu_L(123)=\mu_{L'}(123)=-\mu_{L'}(132)=-m_2\mu_L(132)=m_2\mu_L(123), $$
     and likewise $m_3 \mu_L(123)=m_1\mu_L(123)$. 
     Equation (\ref{taface}) then implies that either $\mu_L(123)=\mu_{L'}(132)=0$ 
     or $m_1=m_2=m_3=\pm 1$. 
     \item If, say, $\mu_L(23)\neq 0$, we have seen above that we directly have $m_2=m_3=\pm 1$. 
     \end{itemize}
     
     Summarizing, in all cases we have that $\Delta_L(123)=\Delta_{L'}(123)$ and, for some $m=\pm 1$, 
 $$m\mu_{L}(12)= \mu_{L'}(12),~ m\mu_{L}(13)= \mu_{L'}(13),~m\mu_{L}(23)= \mu_{L'}(23),$$
 and
$$ m \mu_{L}(123)\equiv \mu_{L'}(123) \mod \Delta_{L}(123).  $$
   
If $m=1$, then $L$ and $L'$ are link-homotopic by Milnor's classification result \cite{Milnor}.  
If $m=-1$, then the same holds for $L$ and the mirror image of $L'$. This concludes the proof
\end{proof}
 Of course the link-homotopy classification for $2$--component Spun links is also implied by Proposition \ref{prop:spun}, for example by adding a third trivial and unlinked component.  
 
For Spun links with an arbitrary number of components, we also have the following consequence of Lemma \ref{lem:Milnorspun}.

\begin{prop}\label{prop:spun2}
Let $L$ be an $m$--component link. The following are equivalent: 
  \begin{itemize}
   \item[(i)] $\Spun(L)$ is link-homotopically trivial; 
   \item[(ii)] $\newmu_{\Spun(L)}(I)=0$\, for any non-repeated sequence $I$;
   \item[(iii)] $L$ is link-homotopically trivial. 
  \end{itemize}
\end{prop}

\begin{proof}
 Implication $(i)\Rightarrow (ii)$ is a consequence of Corollary \ref{cor:lhinvariance}. 
 Condition $(ii)$ implies by Remark \ref{rem:realfirst} that all non-repeated Milnor invariants of $L$ are zero, which implies $(iii)$ according to Milnor \cite{Milnor}. 
 Implication $(iii)\Rightarrow (i)$ is clear, as above: spinning a link-homotopy between $L$ and the trivial link $U$ provides a link-homotopy between $\Spun(L)$ and 
 $\Spun(U)$.
\end{proof}

\subsubsection{Link-homotopy classification of knotted punctured spheres}\label{sec:lhkps}

We next mention another link-homotopy classification result derived from our construction, whose proof is given in a separate paper \cite[\S 18.4]{AMY_w}.

Let $n\ge 2$ be some integer, and $p_1,\ldots,p_n$ some positive integers. 
A \emph{knotted punctured spheres} is any surface-link which is a proper embedding of $\sqcup_{i=1}^n S^2_{\hspace{-.18em}p_i}$ into the $4$--ball $B^4$, 
where $S^2_{\hspace{-.18em}p_i}$ is the $2$--sphere with $p_i$ holes, and such that the boundary is mapped to a fixed trivial link with $\sum_i p_i$ components in $S^3=\p B^4$. 
Notice that taking $p_i=1$ for all $i$ yields the notion of linked disks introduced by Le Dimet in \cite{LeDimet},  
while taking $p_i=2$ for all $i$ produces $2$--string links, as studied in \cite{AMW}. 
We have the following classification result, which generalizes \cite[Thm.~4.8]{AMW}. 
\begin{theo}\label{thm:punk}
  Knotted punctured spheres are classified up to link-homotopy by Milnor invariants.
\end{theo}
Note that knotted punctured spheres do not contain any nontrivial loop longitude, since the boundary is assumed to be a trivial link.  Hence the classifying invariants of the statement are actually Milnor arc-invariants. 

\begin{remarque}
This result can actually be extended to a wider notion of knotted punctured spheres, where the boundary is assumed to form a \emph{slice} link. The slice condition ensures that loop longitudes are still all trivial as nilpotent group elements. A consequence of this fact is that all slice disks for a given slice link are link-homotopic, see \cite[Cor.~18.4.7]{AMY_w}. 
\end{remarque}

\subsection{An obstruction for ribbonness}
\label{sec:Kernel}
Recall that a ribbon immersed $3$--manifold in $4$--space is an immersion such that the only singularities are $2$--disks with two preimages, one embedded in the interior and the other properly embedded. 
A surface-link is \emph{ribbon} if its image is the boundary of some ribbon immersed $3$--dimensional handlebodies in $4$--space  
(note that the latter is in general not unique). 
Ribbon surface-links are the natural analogues of the notion of ribbon knots introduced by Fox in the classical dimension. 
Spun links are examples of ribbon surface-links.  

We now show how Milnor maps can be used to  obstruct ribbonness. 
Let $S$ be the image of an embedding of some surface $\Sigma=\sqcup_{i=1}^n \Sigma_i$ in $4$--space,
and let $Ii$ be some sequence of indices.  

\begin{defi}
The \emph{free kernel} of the Milnor map $M^{Ii}_S: H_1(\Sigma_i)\rightarrow \fract{\Z}{\Delta_\UU(Ii)\Z}$, is the maximal submodule $\Ker_0(M^{Ii}_S)\subset \Ker(M^{Ii}_S)$ 
such that $\fract{H_1(\Sigma_i)}{\Ker_0(M^{Ii}_S)}$ is torsion-free. 
\end{defi}

Since the Milnor map $M^{Ii}$ up to automorphisms of $H_1(\Sigma_i)$ is an invariant for nilpotent peripheral systems up to equivalence (Remark~\ref{rem:MilnorNilpotentMM}), it is a concordance invariant. In particular, the rank of the free kernel is a concordance invariant of surface-links 

Now, suppose that the surface-link $S$ is ribbon.
If the $i$-th component of $S$ has genus $g$, then any ribbon immersed handlebodies $H$ 
bounded by this component induces a rank $g$ submodule
 $\Ker\big(H_1(\Sigma_i)\hookrightarrow H_1(H)\big)$, which is clearly contained in the free kernel of any Milnor map on
 $H_1(\Sigma_i)$. This implies the following obstruction result for ribbonness. 
 \begin{prop}
   \label{prop:RibbonObstruction}
      Let $S$ be a surface-link whose $i$-th component has genus $g$. 
      If there exists a  (nonempty) set $\mathcal{I}$ of sequences ending with $i$ such that 
      $\rk\!\big(\cap_{I\in \mathcal{I}}\Ker_0(M^{I}_S)\big)< g$, then $S$ is not concordant to a ribbon surface-link.  
  Furthermore, if all sequences in $\mathcal{I}$ are
   non-repeated, then $S$ is not even link-homotopic to a ribbon surface-link.
 \end{prop}

Let us give a concrete application. 
For all $m\in\Z$, let $A_{\!m}$ be the surface-link shown on the left-hand side of Figure \ref{fig:NonRibbon}. 
It is obtained by spinning a $3$-component link as illustrated and, while spinning, moving the third unknot component $m$ times around the second one.
\begin{figure}
  \[
\hspace{-1cm}  \dessin{3cm}{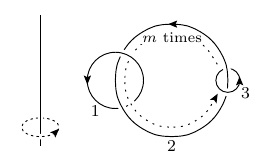}\ ; \ \dessin{4cm}{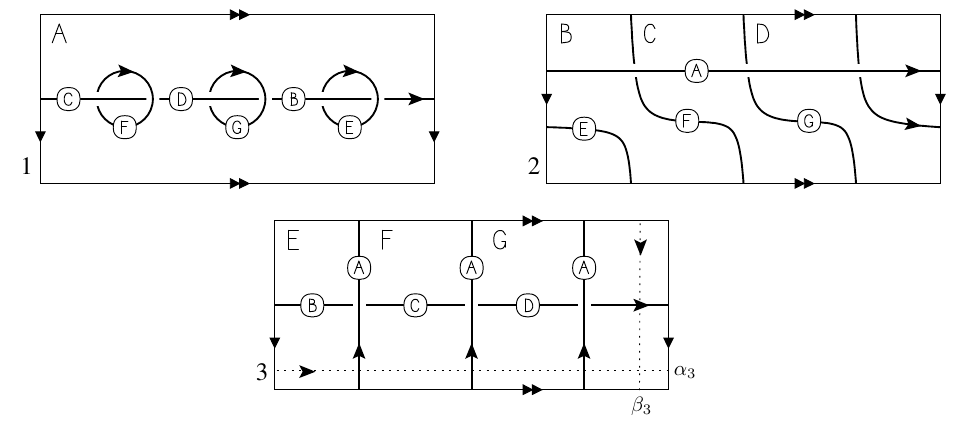}
  \]
  \caption{The surface $A_{\!m}$ (left), and a cut-diagram for $A_{\!3}$ (right).}  \label{fig:NonRibbon}
\end{figure}
 The surface-link $A_{\!1}$ appears in \cite{AMW}, where it was proved that it is not link-homotopic to any ribbon surface. 
Using free kernels, we can prove the following.
\begin{prop}\label{prop:er}
  The surface-link $A_{\!m}$ is link-homotopic to a ribbon one if and only if $m=0$.
\end{prop}
\begin{proof}
  Since $A_{\!0}$ is a Spun link, it is a ribbon surface.  
  Now let $m\neq 0$. 
  For simplicity, we give the proof for $m=3$: the general case is handled in a strictly similar way.   
  The right-hand side of Figure \ref{fig:NonRibbon} 
  gives a cut-diagram $\UU_3$ for the surface-link $A_{\!3}$. 
  We denote by 
  $\alpha_3, \beta_3\in H_1(\Sigma_3)$ the cycles shown in dashed lines 
  on the cut-diagram of Figure \ref{fig:NonRibbon}.
  It is easily computed that
  $\Ker_0(M_{A_{\!3}}^{13})=\langle\beta_3\rangle$ and
  $\Ker_0(M_{A_{\!3}}^{23})=\langle\alpha_3\rangle$, so that
  $\rk\!\big(\Ker_0(M_{A_{\!3}}^{13})\cap \Ker_0(M_{A_{\!3}}^{23})\big)=0$. 
  Proposition \ref{prop:RibbonObstruction} then implies that $A_{\!3}$ cannot be link-homotopic to a ribbon surface.
\end{proof}

Recall from \cite{AMW} that $2$-string links are surface-links in $B^4$ of properly embedded annuli, with boundary a fixed unlink in $S^3$ (see Section \ref{sec:lhkps}). There is a natural $2$-dimensional analogue of the braid closure map, that closes a $2$-string link into knotted tori in $4$--space, see \cite[Rem.~2.2]{ABMW}. 
We have the following noteworthy consequence of Proposition \ref{prop:er}.
\begin{cor}\label{cor:closure}
The braid-closure map from $2$-string links to surface-links of  
tori in $4$--space, is not surjective for $n\ge 3$.
\end{cor}
\begin{proof}
Let $m\neq 0$, and assume by contradiction that the surface-link $A_m$  is the closure of a $3$-component $2$-string link $L$.
By \cite[Thm.3.5]{AMW}, $L$ is link-homotopic to a ribbon $2$-string link $L'$.
Hence $A_m$ is link-homotopic to the closure of $L'$, which is ribbon, a contradiction.
\end{proof}

\subsection{$k$-slice links}\label{sec:k-slice}
The notion of $k$-slice classical links was introduced independently by Orr \cite{Orr2} and 
Cochran \cite{Cochran_AMS}, as follows. 
An $n$-component link $L=K_1\cup\cdots\cup K_n$ in the boundary $\partial D^4$ of 
the $4$--ball $D^4$ is {\em $k$-slice} if there is a disjoint union  
$S=S_1\cup\cdots\cup S_l$ of properly embedded surfaces in $D^4$ with $\partial S_i=K_i~(i=1,\ldots,n)$, 
 such that the composition of the homomorphism $\pi_1(S_i)\longrightarrow \pi_1(D^4\setminus S)$ 
induced by the inclusion from $S_i$ to $D^4\setminus S$ obtained by pushing $S$ in the normal
direction, with the natural projection $\pi_1(D^4\setminus S)\longrightarrow N_k\pi_1(D^4\setminus S)$, is the trivial homomorphism for each $k$. 

Igusa and Orr \cite{IO} showed 
that a link is $k$-slice if and only if all its Milnor invariants of length $\leq 2k$ vanish, 
a result which is known as the {\em $k$-slice conjecture}.

Observe that for each $i$ and all $k$, the composition
\[\pi_1(S_i)\longrightarrow \pi_1(D^4\setminus S)\longrightarrow N_k\pi_1(D^4\setminus S)\]
is trivial if and only if Milnor loop-invariant $\nu_S(Ii)$ vanish for any
sequence $I$ of length $\leq$ $k-1$.  Hence we have the following result, which relates Milnor invariants in dimensions 3 and 4. 
\begin{prop}\label{prop:429}
 A link has vanishing Milnor invariants of  length $\leq 2k$ if and only if it bounds mutually disjoint, smoothly properly embedded once-punctured surfaces in the $4$--ball having vanishing Milnor loop-invariants of length $\leq k$.
\end{prop}

\section{Chen homomorphisms and Chen-Milnor presentations for cut-diagrams}\label{vie}

We give in this section all technical results involving the Chen homomorphisms introduced in Section \ref{sec:ChenMaps}. 
We prove in particular Theorems \ref{alaChen} and Theorem \ref{alaChenSing}. 

Let $\UU$ be a cut-diagram over $\Sigma=\sqcup_{i=1}^n \Sigma_i$, with a choice of basepoint $p_i$ on each connected component $\Sigma_i$. 
Recall that we denote by $R_{ij}$ all regions in $\Sigma_i$, with $R_i:=R_{i0}$ being the region containing $p_i$. 
Recall also the free groups $F=\langle R_i \rangle$ and $\bar{F}=\langle R_{ij} \rangle$. 
Finally, fix a road network $\alpha$ for $\UU$ based at the point $p_i$. Recall that $\alpha_{ij}\in \alpha$
denotes the road running from $p_i$ to region $R_{ij}$, and that $v_{ij}=\widetilde{w}_{\alpha_{ij}}$ denotes the associated word in $\bar{F}$.

\subsection{Some preliminary results}\label{vie1}

We first derive some elementary properties of the Chen homomorphisms $\eta_q$.  We will make use of the following elementary fact.
\begin{claim}\label{claim2}
Let $G$ be a group, $N$ a normal subgroup of $G$, and $q\ge 1$. 
For any $x,y,a\in G$, if $x\equiv y\mod G_q\!\cdot\!N$, then 
$a^x\equiv a^y\mod G_{q+1}\!\cdot\!N$.
\end{claim}
\begin{proof}
This follows from the equivalence $a^x=y^{-1}[xy^{-1},a^{-1}]ay\equiv a^y\mod G_{q+1}$. 
\end{proof}

\begin{lemme}\label{lem1}
 For all $q\ge 1$, and all $w\in\bar{F}$, $\eta_{q}(w)\equiv \eta_{q+1}(w)$ mod $F_q$. 
 \end{lemme}
\begin{proof}
The claim is obvious for $q=1$, since $F_1=F$. 
Assume inductively that the claim holds for some $q\ge 1$. Then it suffices to consider the case $w=R_{ij}$ for any $i,j$.  By induction
hypothesis, $\eta_{q+1}(v_{ij})\equiv \eta_q(v_{ij})\mod F_q$ so that, by Claim \ref{claim2} and the definition of $\eta_{q+1}$, $\eta_{q+1}(R_{ij}) =
R_i^{\eta_q(v_{ij})}\equiv
R_i^{\eta_{q+1}(v_{ij})}=\eta_{q+2}(R_{ij})\mod F_{q+1}.$
\end{proof}

Recall from Notation \ref{nota:W} that $W$ is the normal subgroup of $\bar{F}$ generated by all Wirtinger relations in $G(\UU)$. 
\begin{lemme}\label{lem2}
For any
generator $R_{ij}$ of $\bar{F}$, $\eta_q(R_{ij})\equiv R_{ij}$ mod $\bar{F}_q\!\cdot \!W$. 
\end{lemme}
\begin{proof}
We proceed by induction on $q$. The case $q=1$ is trivial. 
Assume now  that for some $q\ge 1$ we have $\eta_q(R_{ij})\equiv R_{ij}$ mod $\bar{F}_q\!\cdot \!W$ for any $i,j$. 
This implies that $\eta_{q}( v_{ij}^{\pm 1})\equiv v_{ij}^{\pm 1}$ mod $\bar{F}_q\!\cdot \!W$. 
Hence by Claim \ref{claim2} and the definition of $\eta_{q+1}$ we have 
that $\eta_{q+1}(R_{ij}) =R_i^{\eta_q(v_{ij})}\equiv R_i^{v_{ij}}\equiv R_{ij}$ mod $\bar{F}_{q+1}\!\cdot \!W$. 
\end{proof}

\subsection{Nilpotent quotients}\label{vie2}

We first give here a presentation for the nilpotent quotients of $G(\UU)$   
with as many generators as components in $\Sigma$, and only nilpotent and meridian/loop-longitude commutation relations. 
The Chen--Milnor presentation of Theorem \ref{alaChen}, proved in the next subsection, will significantly reduce the number of relations. 

To do so, we  define
  \emph{walls} of
  $\UU$, which are connected components of regular points of
  $\UU$. They actually correspond to regular \quote{walls} between adjacent
  regions. Note that each wall is included in a cut-arc of $\UU$, hence inherits a labeling by some region.
  
Consider some wall $a$ of $\UU$, say on the $i$-th component and labeled by some region $R$.
Recall from Section \ref{sec:PeripheralSystems} that the Wirtinger relation 
$R_{ik}=R_{ij}^R$ holds in $G(\UU)$, 
where $(R_{ij},R_{ik})$ is the pair of $R$--adjacent regions at $a$. 
We define the word $w_a\in \bar{F}$ by 
\begin{equation}
w_a:=v_{ij}Rv^{-1}_{ik},
\label{eq:w_a}
\end{equation}
which is the (not necessarily normalized) word 
associated to the $p_i$--based loop defined
by connecting the road $\alpha_{ij}$ to $\alpha_{ik}^{-1}$ by a segment crossing $a$ transversally.

\begin{nota}\label{nota:C}
For all $q\ge 1$, we denote by $C_{(q)}$ the normal closure of all relations $[R_i,\eta_{q}(w_a)]$ in $F$, for all 
walls $a$ in $\UU$, with $w_a$ defined in Equation (\ref{eq:w_a}) above.   
\end{nota}

\begin{lemme}\label{lem3}
 For all $q\ge 1$, we have an isomorphism 
 \[N_qG(\UU)\cong  \fract{F}{F_q\!\cdot\! C_{(q)}}.\] 
 In other words, $N_qG(\UU)$ has the following presentation 
\[
\left\langle R_1,\ldots,R_n\ \ \Bigg|
  \begin{array}{l}
    F_q\,;\, \big[R_i,\eta_q(w_a)\big]
\textrm{ for all $i$ and all walls $a$ on $\UU\cap \Sigma_i$} 
  \end{array}
\right\rangle.
\]
 \end{lemme}
\begin{proof}
The homomorphism $\eta_q$ naturally induces a map $N_q\bar{F}\rightarrow \fract{F}{F_q\!\cdot\! C_{(q)}}$.  
Consider, with the above notation, the Wirtinger relator $R_{ik}^{-1}R_{ij}^R$ at some wall $a$ on $\UU\cap \Sigma_i$. 
We have
\begin{align*}
\eta_q\left(R_{ik}^{-1}R_{ij}^R\right) 
& =
  \left(R_i^{\eta_{q-1}(v_{ik})}\right)^{-1}\eta_q(R)^{-1}R_i^{\eta_{q-1}(v_{ij})}\eta_q(R)
  \\
& \equiv
  \left(R_i^{\eta_q(v_{ik})}\right)^{-1}\eta_q(R)^{-1}R_i^{\eta_q(v_{ij})}\eta_q(R)\
  =\ \big[ R_i,\eta_q(w_a)\big]^{\eta_q(v_{ik})}\mod
  F_q
\end{align*}
where the equivalence combines Lemma \ref{lem1} and Claim \ref{claim2}. 
This shows that $\eta_q$ induces a well-defined epimorphism
from $N_qG(\UU)=N_q\left(\fract{\bar{F}}{W}\right)\simeq \fract{\bar{F}}{\bar{F}_q\cdot W}$ to $\fract{F}{F_q\!\cdot\! C_{(q)}}$.
 It is easily checked, using Lemma \ref{lem2} and the above equivalence, that the natural
inclusion map $F\rightarrow \bar{F}$ yields an inverse for this epimorphism, which is thus an isomorphism. 
\end{proof}

We have the following seemingly weaker form of Lemma \ref{lem3},
  with an infinite number of commutator relations. This version shall however be useful. 
\begin{prop}\label{cor:InfiniteRelations}
For all $q\geq1$, the group $N_qG(\UU)$ has the following presentation
\[
\left\langle R_1,\ldots,R_n\ \ \Bigg|
  \begin{array}{l}
    F_q\,;\,\big[R_i,\eta_q(w_\gamma)\big]
\textrm{ for all $i$ and all $p_i$--based generic loop $\gamma$
    on $\Sigma_i$}
  \end{array}
\right\rangle.
\]
\end{prop}

\begin{remarque}\label{rem:normornotnorm}
We may equally well use relations $\big[R_i,\eta_q(\widetilde{w}_\gamma)\big]$, with the 
words $\widetilde{w}_\gamma$ in the above presentation.  Indeed we have, for some $s\in \mathbb{Z}$: 
\[
 \big[R_i,\eta_q(w_\gamma)\big]=\big[R_i,\eta_q(R_i^s\widetilde{w}_\gamma)\big]=\big[R_i,R_i^s\eta_q(\widetilde{w}_\gamma)\big]=\big[R_i,\eta_q(\widetilde{w}_\gamma)\big].
\]
\end{remarque}

\begin{proof}
Using Remark \ref{rem:normornotnorm}, we have that relations in Lemma \ref{lem3} are special cases of the
  relations in Proposition \ref{cor:InfiniteRelations}. It is hence sufficient to show that, for any $p_i$--based loop $\gamma$
    on $\Sigma_i$, the relation $\big[R_i,\eta_q(w_\gamma)\big]$ can be deduced from the relations in 
    $$\Big\{ \big[R_i,\eta_q(w_a)\big]\ \big|\  \textrm{$a$ wall of $\UU$}\Big\}. $$ 
  The proof is essentially given in the figure below, as follows. 
The word $w_{\gamma}$ is of the form 
\[w_\gamma:=R_{i}^{s}R_{i_1j_1}^{\varepsilon_1}R_{i_2j_2}^{\varepsilon_{2}}\cdots R_{i_mj_m}^{\varepsilon_{m}}, \]
for some $m,s\ge 0$ and $\varepsilon_k\in \{\pm 1\}$, where the $R_{i_kj_k}\in \bar{F}$ are 
the labels of the successive walls $a_k$ met transversally when
running along $\gamma$; see below on the left. 
\[
\dessin{3cm}{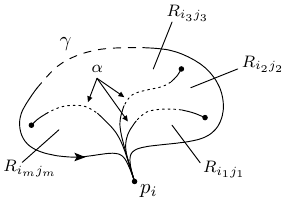} \ \ \leadsto\ \ \dessin{3cm}{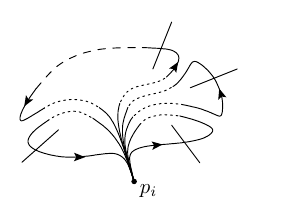}
\]
Then, as illustrated in the figure, a homotopy of $\gamma$, guided by
the road network $\alpha$ and involving only moves $\textnormal H_1$
of Lemma \ref{lem:GenMoves},
gives directly the decomposition
\[
w_\gamma=R_{i}^{s}w_{a_1}^{\varepsilon_1}w_{a_2}^{\varepsilon_2} \cdots
w_{a_r}^{\varepsilon_r},
\]
where the words $w_{a_k}$ are defined in Equation (\ref{eq:w_a}).
The commutation of
$R_i$ with $\eta_q(w_\gamma)$ then follows easily from the commutation of $R_i$ with
each $\eta_q(w_{a_k})$.
\end{proof}

\subsection{Chen--Milnor presentation for nilpotent quotients: proof of Theorem \ref{alaChen}}\label{vie3}

Recall that we pick, for each component $\Sigma_i$ of $\Sigma$, a system of loop-longitudes $\mathcal{L}_i(\Sigma):=\{w_{ij}\}_j$ associated with a collection of $p_i$-based loops $\{\gamma_{ij}\}_j$ comprising a generating set for $\pi_1(\Sigma_i;p_i)$.  

The proof of Theorem \ref{alaChen} builds on the presentation given in Proposition \ref{cor:InfiniteRelations}. The idea is to show that the infinite family of commutation relations  in Proposition \ref{cor:InfiniteRelations}, involving \emph{all} generic loops on $\Sigma_i$, can be realized by just the commutation relations involving the finite collection of loops in $\{\gamma_{ij}\}_i$. 
In fact, we shall prove more precisely that $F_q\cdot L_{(q)}=F_q\cdot V_{(q)}$, 
where
\begin{itemize}
\item $L_{(q)}$ is the normal closure of
  $\left\{[R_i,\eta_q(w_\gamma)]\ | \ \textrm{$\gamma$ any
      $p_i$--based generic loop on $\Sigma_i$}\right\}$ in $F$, 
\item $V_{(q)}$  is  the normal closure of
  $\left\{[R_i,\eta_{q}(w_{ij})]\ | \ w_{ij}\in \mathcal{L}_i(\Sigma) \right\}$ in $F$. 
\end{itemize}

The inclusion $F_q\!\cdot\!V_{(q)}\subset F_q\!\cdot\!L_{(q)}$ is clear. 
To prove the converse inclusion, 
it suffices to show that the map 
$\fract{F}{ F_q\!\cdot\!V_{(q)}}
\longrightarrow \fract{F}{ F_q\!\cdot\!L_{(q)}}$, 
induced by the identity, is injective. 
This is a consequence of Proposition \ref{claimk} below.
Indeed, since any loop on $\Sigma_i$ is homotopic to a product of elements in $\{\gamma^{\pm 1}_{ij}\}_j$,
this proposition in particular implies that $[R_i,\eta_{q}(w_\gamma)]\equiv 1$ mod $F_{q}\!\cdot\!V_{(q)}$ 
for any $p_i$--based loop $\gamma$ on $\Sigma_i$, and Theorem \ref{alaChen} follows. 

\begin{prop}\label{claimk}
For any two homotopic paths  $\gamma,\gamma'$ on $\Sigma$, we have  
\[\eta_{q}(w_\gamma)\equiv \eta_{q}(w_{\gamma'})\textrm{ mod $F_{q}\!\cdot\!V_{(q)}$}.\]
\end{prop}
 We stress that the paths $\gamma,\gamma'$ in the statement may have arbitrary endpoints, and are in particular not required to start at the region $R_i$.
\begin{proof}
Note that the case $q=1$ of the claim is trivial, 
and assume inductively that Proposition \ref{claimk} holds for some $q\ge 1$. 
We first observe the inclusion 
\begin{equation}\label{tiroir}
F_{q+1}\!\cdot\!V_{(q)}\subset F_{q+1}\!\cdot\!V_{(q+1)}\textrm{, for all $q\ge 1$.}
\end{equation}
This is true, since by Lemma \ref{lem1} and Claim~\ref{claim2}, we have  $[a_i,\eta_{q}(w_{ij})]\equiv [a_i,\eta_{q+1}(w_{ij})]$ mod $F_{q+1}$ for any $i,j$.

Now, two homotopic paths  $\gamma,\gamma'$ differ by a sequence of 
 the three local moves $\textnormal H_1$, $\textnormal H_2$ and $\textnormal H_3$ of  Lemma \ref{lem:GenMoves}. 
So it suffices to check the claim when $\gamma$ and $\gamma'$ differ by one of these three moves.
In the rest of this proof, for a given region $R$, we shall denote by $v_R$ the 
word associated with the road from the road network $\alpha$ running to $R$.
\begin{itemize}
\item If $\gamma$ and $\gamma'$ differ by move $\textnormal H_1$, then
  $w_{\gamma}=w_{\gamma'}$ in $\bar{F}$ so the desired equivalence holds trivially. 
\item
Suppose that $\gamma$ and $\gamma'$ differ by move
$\textnormal H_2$, with $\gamma$ corresponding to the left-hand side
of the move. We then have  
\[w_{\gamma}= B^{s}\widetilde{w}_{\gamma_1}\widetilde{w}_{\gamma_2}
\quad\textrm{and}\quad 
w_{\gamma'}=B^{s-\varepsilon}\widetilde{w}_{\gamma_1}A^{\varepsilon}\widetilde{w}_{\gamma_2},\] 
where $\gamma_1,\gamma_2$ are paths such that
  $\gamma=\gamma_1.\gamma_2$, $s\in \Z$, $\varepsilon={\pm 1}$, and
  $B$ is the region of $\Sigma_i$ where $\gamma_1$ starts:
\[
\dessin{2cm}{Cas2}.
\]
By definition of the Chen homomorphism, we have
\[
\eta_{q+1}(w_{\gamma'})
= \eta_{q}(v_B^{-1})R_{i}^{s-\varepsilon}\left( R_i^{\varepsilon}\right)^{\eta_{q}(v_A) \eta_{q+1}(\widetilde{w}_{\gamma_1}^{-1})
\eta_{q}(v_B^{-1})}\eta_{q}(v_B)\eta_{q+1}(\widetilde{w}_{\gamma_1}\widetilde{w}_{\gamma_2}). \]
Observe that 
$v_A \widetilde{w}_{\gamma_1}^{-1} v_B^{-1}$ represents a $p_i$--based loop on $\Sigma_i$, 
which is homotopic to a product $\xi$ of elements in $\{\gamma_{ij}^{\pm1}\}$.
Using Lemma \ref{lem1} and the induction hypothesis, we thus have 
\[
\eta_{q}(v_A)
\eta_{q+1}(\widetilde{w}_{\gamma_1})^{-1}\eta_{q}(v_B^{-1}) \equiv
\eta_{q}(v_A \widetilde{w}_{\gamma_1}^{-1} v_B^{-1}) \equiv
\eta_{q}(\widetilde{w}_\xi)\mod F_{q}\!\cdot\!V_{(q)}.
\]
This implies, by Claim~\ref{claim2}, that 
\[ \quad\quad\quad \eta_{q+1}(w_{\gamma'})\equiv  
\eta_{q}(v_B^{-1})R_{i}^{s-\varepsilon}\left( R_i^{\varepsilon}\right)^{\eta_{q}(\widetilde{w}_{\xi})}\eta_{q}(v_B)\eta_{q+1}(\widetilde{w}_{\gamma_1}\widetilde{w}_{\gamma_2})
\equiv 
\eta_{q+1}(w_{\gamma}) \mod F_{q+1}V_{(q)}, \]
where the last equivalence simply uses the fact that $[R_i,\eta_q(\widetilde{w}_\xi)]=[R_i,\eta_q(w_\xi)]\equiv 1\mod V_{(q)}$.
The conclusion then follows from (\ref{tiroir}). 
 \item Suppose now that $\gamma$ and $\gamma'$ differ by move $\textnormal H_3$,
   with $\gamma$ corresponding to the left-hand side of the move. 
Then, up to some moves $\textnormal H_1$, we may assume that the
situation is as shown below, so that there exist two arcs $\gamma_1,\gamma_2$ on $\Sigma_i$ such that
\[
w_\gamma = D^s\widetilde{w}_{\gamma_1}\widetilde{w}_{\gamma_2}
 \quad\textrm{and}\quad  
 w_{\gamma'} = D^s\widetilde{w}_{\gamma_1}(A^{\varepsilon})^{C}B^{-\e}\widetilde{w}_{\gamma_2},
\]
for some $s\in \Z$ and some $\varepsilon\in\{\pm 1\}$,
where $D$ is the starting region of $\gamma_1$ and 
$(A,B)$ is the pair of $C$--adjacent cut-arcs involved in the move, 
which lie on the $j$-th component of $\Sigma$ for some $j$. 
\[
\dessin{2cm}{Cas3}
\]
Using the definition of Chen homomorphisms, we have
\[ \eta_{q+1}(w_{\gamma'}) = 
\eta_{q+1}(D^s\widetilde{w}_{\gamma_1})[X,R_j^{-\e}]^{\eta_q(v_B)}\eta_{q+1}(\widetilde{w}_{\gamma_2}), 
 \]
where we set $X:=\eta_q(v_A)\eta_{q+1}(C)\eta_q(v_B^{-1})$. 
By Lemma \ref{lem1}, we have $X\equiv \eta_q(v_ACv_B^{-1}) \mod F_q$.
Now, the word $v_ACv_B^{-1}$ is associated to a $p_j$--based loop on $\Sigma_j$, which is homotopic to a product $\xi$ of elements in $\{\gamma_{jk}^{\pm1}\}$. So the induction hypothesis says  that
$ X\equiv \eta_{q}(\widetilde{w}_\xi) \mod F_{q}\!\cdot\!V_{(q)}$.
Claim~\ref{claim2} then gives 
\[
[X,R_j^{-\e}]=(R_j^\e)^XR_j^{-\e}\equiv
(R_j^\e)^{\eta_{q}(\widetilde{w}_\xi)}R_j^{-\e}\equiv 1\mod
F_{q+1}\cdot V_{(q)}. 
\]
By (\ref{tiroir}), this implies that 
\[ 
 \eta_{q+1}(w_{\gamma'}) \equiv  
\eta_{q+1}(D^s\widetilde{w}_{\gamma_1}) \eta_{q+1}(\widetilde{w}_{\gamma_2})
\equiv  \eta_{q+1}(w_{\gamma})\mod  F_{q+1}\cdot V_{(q+1)}.
\]
\end{itemize}
This concludes the proof. 
\end{proof}

\subsection{Chen--Milnor presentation for reduced quotients: proof of Theorem \ref{alaChenSing}}\label{vie4}

The presentation for R$G(\UU)$ given in Theorem \ref{alaChenSing} is obtained in a completely similar way as Theorem \ref{alaChen} in the previous subsection. 
Hence we just sketch the argument below, stressing only the new ingredients. 

We first note that the Chen homomorphisms $\eta_q$ induce well-defined homomorphisms $\textrm{R}\bar{F} \rightarrow \textrm{R}F$. 
The arguments proving Lemma \ref{lem3} and Proposition \ref{cor:InfiniteRelations}
then apply verbatim to show that, for any $q\ge 1$, the $q$-th nilpotent quotient of $\nR G(\UU)$  has the following presentation: 
\begin{equation}
\label{eq:SingPresentation}
N_q \nR G(\UU) \cong \left\langle R_1,\ldots,R_n\ \ \left|
  \begin{array}{l}
    F_q\\[.0cm] 
    [R_i,R_i^g]\textrm{ for all $i$ and all $g\in F$}\\[.0cm]
    [R_i,\eta_q(w_\gamma)]
\textrm{ for all $i$ and all $p_i$--based generic loop $\gamma$ on $\Sigma_i$}
  \end{array}
\right.\right\rangle.
\end{equation}

From this point, we can use an induction on $q$ to rework the
presentation so that the third type of relations are substituted by commutation relations of the form $\big[R_i,\eta_q(w_{ij})\big]$,  as in the proof of Theorem \ref{alaChen}.
The only new ingredient here is that move $\textnormal H_2$ from Lemma \ref{lem:GenMoves} may now involve a $\circ$ endpoint, that is, a terminal cut-arc which is not labeled by the region supporting the move, 
but by any region in the same connected component. 
Suppose hence that two arcs $\gamma$ and $\gamma'$ on $\Sigma_i$ differ by such a move $\textnormal H_2$, with $\gamma$ corresponding to the
left-hand side of the move (in the figure of Lemma \ref{lem:GenMoves}).
By the labeling condition 2') of Definition \ref{def:self-cut}, the cut-arc involved in this move
is labeled by some region $A$ of $\Sigma_i$. We then have  
\[w_{\gamma}= B^{s}\widetilde{w}_{\gamma_1}\widetilde{w}_{\gamma_2}
\quad\textrm{and}\quad 
w_{\gamma'}=B^{s-\varepsilon}\widetilde{w}_{\gamma_1}A^{\varepsilon}\widetilde{w}_{\gamma_2},\] 
where $\gamma_1,\gamma_2$ are paths such that
$\gamma=\gamma_1.\gamma_2$, $s\in \Z$, $\varepsilon={\pm 1}$, and $B$
is the region of $\Sigma_i$ where $\gamma_1$ starts. 
As in the proof of Proposition \ref{claimk}, denote by $v_R$, for any region $R$,
the word associated to the road of $\alpha$ running to 
the region $R$. We have, by the definition of the Chen homomorphisms:  
\begin{equation}\label{ouiche}
\eta_{q+1}(w_{\gamma'})
= \left(R_{i}^{-\varepsilon}\right)^{\eta_{q}(v_B)} \left( R_i^{\varepsilon}\right)^{\eta_{q}(v_A) \eta_{q+1}(\widetilde{w}_{\gamma_1}^{-1}B^{-s})}\eta_{q+1}(\widetilde{w}_{\gamma}). 
\end{equation}
Hence $\eta_{q+1}(w_{\gamma'})$ and $\eta_{q+1}(w_{\gamma})$ differ by a product of conjugates of $R_i$. This implies that 
$[R_i,\eta_{q+1}(w_{\gamma'})]$ and $[R_i,\eta_{q+1}(w_{\gamma})]$ are equivalent modulo the normal closure of the relations $[R_i,R_i^g]$ for all $i$ and all $g\in F$.

Now, as already noted in Section \ref{sec:loug}, Habegger--Lin's result \cite[Lem. 1.3]{HL} implies that $\nR
G(\UU)$ is a nilpotent group of order at most $n$. 
Picking any $q\geq n$, it follows that
  (\ref{eq:SingPresentation}) provides a presentation for
  $\nR G(\UU)\cong N_q \nR G(\UU)$, and also that the relations involving
  $q$-th iterated commutators are trivially satisfied. We thus obtain the desired presentation and the proof is complete. 

\subsection{Chen homomorphisms and basepoint change}\label{vie5}

We now consider the situation where 
we pick a new set of basepoints $\{p'_i\}$ for $\Sigma$. This induces a new set of meridians $\{R'_i\}$, 
which provides a new presentation for $N_qG(\UU)$ by Theorem \ref{alaChen}. 
The purpose of this subsection is to understand explicitly how a given word in the alphabet $\big\{R_i^{\pm 1}\big\}$, representing some element of $N_qG(\UU)$, translates into a word in this new alphabet  $\big\{{R'_i}^{\pm 1}\big\}$. 
This is done in Corollary \ref{rem:baze}, where the Chen homomorphisms provide the desired dictionary between the two presentations of $N_qG(\UU)$.
We will use the following.
\begin{nota}
Given an element $x$ in the free group generated by a set $\{e_1,\ldots,e_n\}$, we denote 
by $x_{\scalebox{0.85}{$[e_i\mapsto f_i]$}}$ the element in a given  group $G$ obtained by substituting each $e_i$ by $f_i\in G$. 
\end{nota}
Now, let $\alpha$ be a road network for $\UU$ based at $\{p_i\}$, and 
let $\alpha'$ be another choice of road network, based at  $\{p'_i\}$. 
Consider in $\alpha'$ the road $a_i$ running from $p'_i$ to the basepoint $p_i$ of $\alpha$, and set $v_i:=\widetilde{w}_{a_i}$, the associated (unnormalized) word. 
The maps $\eta^\alpha_q$ take values in the free group $F=\langle R_i \rangle$, 
while the maps $\eta^{\alpha'}_q$ take values in the free group $F'=\langle R'_i \rangle$, 
and both can be seen as taking values in the free group $\bar{F}=\langle R_{ij} \rangle$ generated by \emph{all} regions. 

Recall that $W$ is the normal subgroup of $\bar{F}$ generated by all
Wirtinger relations, see Notation \ref{nota:W}. 
Then for any $w\in \bar{F}$, one can see $\eta^\alpha_q(w)$ as an element of $\fract{F}{F_q\!\cdot\!\eta^\alpha_q(W)}$, while 
 $\eta^{\alpha'}_q(w)$ is seen as an element of $\fract{F'}{F'_q\!\cdot\!\eta^{\alpha'}_q(W)}$. 
\begin{lemme}\label{lem:baze}
In the above notation, for any $w\in\langle R_{ij} \rangle$,  we have 
\[ \eta^{\alpha'}_q(w) = \eta^{\alpha}_q(w)_{\scalebox{0.85}{$\big[ R_i\mapsto {{R'_i}}^{\,\eta^{\alpha'}_{q}(v_i)}\big]$}}\textrm{ in $\fract{F'}{F'_q\!\cdot\!\eta^{\alpha'}_q(W)}$.}\]
\end{lemme}
 This lemma indeed provides the desired dictionary between our two presentations, as follows.  
\begin{cor}\label{rem:baze}
Let $\{R_i\}$ and $\{R'_i\}$ be two sets of meridians for some cut-diagram $\UU$ over $\Sigma$. 
 Let $w$ be a word of $R_i^{\pm1}$ representing some element $\omega$ in $N_qG(\UU)$. 
 Then $\omega$ is represented by the word of ${{R'_i}}^{\pm1}$ obtained from $w$ by replacing each $R_i$ by ${{R'_i}}^{\,\eta^{\alpha'}_{q}(v_i)}$. 
\end{cor}
\begin{proof}
Let us regard $w$ as a word of $\big\{R_{ij}^{\pm 1}\big\}$ involving only the letters $R_i^{\pm1}$---any element of $F$ can be seen as sitting in $\langle R_{ij} \rangle$ in this way. Then we have by definition of the Chen homomorphisms that $w=\eta^\alpha_q(w)$. 
But we also have that $\eta^{\alpha'}_q(w)$ is a representative word of $\big\{{R'_i}^{\pm 1}\big\}$. 
Lemma \ref{lem:baze} tells us that this representative is obtained by replacing each $R_i$ by ${{R'_i}}^{\,\eta^{\alpha'}_{q}(v_i)}$ in the word $w$. 
\end{proof}

\begin{proof}[Proof of Lemma \ref{lem:baze}]
By definition of the homomorphism $\eta^\alpha_q$, 
and regarding $\eta^{\alpha'}_q(w)$ as an element of $\bar{F}$, we have
\[
\eta^{\alpha'}_q(w)\equiv w_{\scalebox{0.85}{$\left[R_{ij}\mapsto
\eta^{\alpha'}_q(R_{ij}) \right]$}}\mod\bar{F}_q\!\cdot\!W.
\]
Now, a direct consequence of Lemma \ref{lem2}
 is that, for any $i,j$, we have
 \[
\eta^{\alpha'}_q\left( \eta^{\alpha}_q(R_{ij}) \right) \equiv
\eta^{\alpha'}_q(R_{ij})\equiv R_{ij}\mod \bar{F}_q\!\cdot\!W.
\]
There is hence some element $g\in \bar{F}_q\!\cdot\!W$ such that 
\[
w_{\scalebox{0.85}{$\left[R_{ij}\mapsto \eta^{\alpha'}_q(R_{ij})\right]$}} = w_{\scalebox{0.85}{$\left[R_{ij}\mapsto
\eta^{\alpha'}_q\left( \eta^{\alpha}_q(R_{ij})\right)\right]$}}g. 
\]
But since $F'$ is a subgroup of $\bar{F}$, one can compose
 $\eta_q^{\alpha'}$ with itself, and it follows from the definition 
 that  $\eta_q^{\alpha'}\circ  \eta_q^{\alpha'} =  \eta_q^{\alpha'}$. 
Applying $\eta^{\alpha'}_q$ to the above equality thus gives: 
\[
w_{\scalebox{0.85}{$\left[R_{ij}\mapsto \eta^{\alpha'}_q(R_{ij})\right]$}} 
 = w_{\scalebox{0.85}{$\left[R_{ij}\mapsto \eta^{\alpha'}_q( \eta^{\alpha}_q(R_{ij}))\right]$}}\eta_q^{\alpha'}(g). \]
This shows that the word
$\eta^{\alpha'}_q(w)= w_{\scalebox{0.85}{$\left[R_{ij}\mapsto \eta^{\alpha'}_q(R_{ij})\right]$}}$ in $\fract{F'}{F'_q\!\cdot\!\eta^{\alpha'}_q(W)}$
is obtained from 
$\eta^{\alpha}_q(w)= w_{\scalebox{0.85}{$\left[R_{ij}\mapsto \eta^{\alpha}_q(R_{ij})\right]$}}$ in $\fract{F}{F_q\!\cdot\!\eta^{\alpha}_q(W)}$
by substituting each $R_i\in F$ by $\eta^{\alpha'}_q(R_i)$.
Now, by the inductive definition of $\eta^{\alpha'}_q$, we have that $\eta^{\alpha'}_q(R_{i})={R'_i}^{\,\eta^{\alpha'}_{q-1}(v_i)}$, 
and Lemma \ref{lem1}
 tells us that $\eta^{\alpha'}_{q-1}(v_i)\equiv \eta^{\alpha'}_{q}(v_i)\mod F'_{q-1}$. 
The conclusion then follows from Claim \ref{claim2}. 
\end{proof}

\bibliographystyle{abbrv}
\bibliography{References}

\end{document}